\documentclass[a4paper,10pt]{article}
\usepackage[utf8]{inputenc}
\usepackage[a4paper, total={7in, 9in}]{geometry}
\usepackage{microtype}
\usepackage{amssymb}
\usepackage{amsmath}
\usepackage{subcaption}
\usepackage{graphicx,amsthm,color}
\usepackage{diagbox}
\usepackage{float}
\usepackage{hyperref,cleveref,cite,url}
\hypersetup{
  colorlinks,
  citecolor=green!50!black,
  linkcolor=red,
  urlcolor=blue!50!black}
\usepackage{makecell}
\usepackage{pgf,tikz}
\usetikzlibrary{arrows}
\usepackage{authblk}
\usepackage{xspace}
\usepackage{enumerate}

\newtheorem{theorem}{Theorem}
\newtheorem{question}{Question}
\newtheorem{proposition}{Proposition}

\newtheorem{corollary}{Corollary}
\newtheorem{lemma}{Lemma}
\newtheorem{claim}{Claim}

\theoremstyle{definition}

\newtheorem{definition}{Definition}

\DeclareMathOperator{\Imp}{imp}
\DeclareMathOperator{\gate}{gate}

\DeclareMathOperator{\conv}{conv}
\DeclareMathOperator{\Star}{St}
\DeclareMathOperator{\Face}{F}
\DeclareMathOperator{\Cell}{C}

\newcommand{\M}{\ensuremath{\mathcal{M}}}

\newcommand{\covectors}{\ensuremath{\mathcal{L}}}

\newcommand{\topes}{\ensuremath{\mathcal{T}}}
\DeclareMathOperator{\Sep}{\mathrm Sep}
\newcommand{\rk}{\ensuremath{{\rm rank}}}

\newcommand{\lgate}{{\langle\langle}}
\newcommand{\rgate}{{\rangle\rangle}}

\newcommand{\cM}{\mathcal{M}\xspace}

\newcommand{\ou}{\overline{u}\xspace}
\newcommand{\oz}{\overline{z}\xspace}
\newcommand{\ov}{\overline{v}\xspace}
\newcommand{\ox}{\overline{x}\xspace}

\newcommand{\KK}[1]{\textcolor{olive}{KK: #1}}

\title{Cell structure of mediangle graphs}
\author{Victor Chepoi$^{1,4}$, Anthony Genevois$^2$, and Kolja Knauer$^3$}

\date{\today}

\begin{document}

\maketitle

	\bigskip
	\centerline{$^{1}$LIS, Aix-Marseille Universit\'e, CNRS, and Universit\'e
	de Toulon}
	\centerline{Facult\'e des Sciences de Luminy, F-13288 Marseille Cedex 9,
	France}
	\centerline{ {\sf victor.chepoi@lis-lab.fr} }

	\bigskip
	\centerline{$^{3}$Institut Math\'ematiques Alexander Grothendieck,}
	\centerline{Universit\'e de Montpellier, France}
    \centerline{ {\sf anthony.genevois@umontpellier.fr} }

	\bigskip
	\centerline{$^{4}$Departament de Matem\`atiques i Inform\`atica,}
	\centerline{Universitat de Barcelona (UB),}
    \centerline{Centre de Recerca Matemàtica, Barcelona, Spain}
    \centerline{ {\sf kolja.knauer@ub.edu} }
    
   \bigskip
    \centerline{$^{4}$ Institut Universitaire de France (IUF)}

\begin{abstract}  Mediangle graphs are a common generalization of median graphs (1-sekeleta of CAT(0) cube complexes) and Coxeter graphs (Cayley graphs of Coxeter systems).
Answering a question motivated from geometric group theory, we show that these graphs can be endowed with the structure of a contractible cell complex. We further show that the cells of this complex are products of simplices and simplicial oriented matroids. A crucial part of the proof identifies bipartite mediangle graphs as tope graphs of finitary Complexes of Oriented Matroids.
\end{abstract}

\tableofcontents

\section{Introduction} 

\subsection{Avant-propos} Genevois \cite{Ge_mediangle,Ge_rotation} introduced mediangle graphs as a common generalization of median graphs (1-sekeleta of CAT(0) cube complexes), Coxeter graphs (Cayley graphs of Coxeter systems), and quasi-median graphs (Hamming analogs of median graphs). He established important and interesting properties of mediangle graphs and groups acting on them and presented other examples of mediangle graphs, arising in geometric group theory and in metric graph theory. He exhibited the $2$-dimensional cell structure of mediangle graphs, induced by their convex cycles and asked whether, analogously to median and Coxeter graphs, mediangle graphs could be endowed with a (high-dimensional) cell structure  of a contractible cell complex. In this paper, we answer  this question in the affirmative by proving that the cell complex $X(G)$, having the gated-antipodal  subgraphs of a mediangle graph $G$ as cells, is contractible.  As a main ingredient we prove that bipartite mediangle graphs are tope graphs of Complexes of Oriented Matroids (COMs), thus  providing a novel connection between objects in Geometric Group Theory and the theory  of Oriented Matroids.  The cells of a COM   are Oriented Matroids (OMs). We show that the cells of the COM arising from a bipartite mediangle graph are the simplicial OMs, a class that has been in the focus of researchers due to its lattice properties~\cite{BjEdZi} and relation to simplicial central hyperplane arrangements~\cite{De}. Finally, we prove that the cells in general mediangle graphs are products of simplices (cliques) and simplicial OMs.  

\emph{Mediangle graphs} are defined by two global metric conditions and two local conditions. 
Recall that the \emph{interval} $[u,v]$ between two vertices $u,v$ of a connected graph consists of all vertices $w$ on $(u,v)$-shortest paths. The \emph{Cycle Condition} requires that  for each pair of vertices $z,u$ and any pair of neighbors $x,y\in [z,u]$ of $z$, there exists a convex cycle $C\subseteq [z,u]$  containing $z,x,y$ and such that the vertex $\oz$ opposite to $z$ in $C$ belongs to $[x,u]\cap [y,u]$. Denoting by $d$ the graph distance, the \emph{Triangle Condition} requires that for each triplet $x,y,u$ such that $d(x,y)=1, d(x,u)=d(x,v)=k$, there exists $w$ adjacent to $x$ and $y$ such that $w\in [x,u]\cap [y,u]$.
The first local condition ensures that the convex cycles define a polygonal complex in which the intersection of any two even convex cycles is either empty, a vertex, or an edge. We call it the \emph{Intersection of Even Cycles Condition}. Finally, the \emph{Diamond Condition} requires the graph to be diamond-free (a diamond is a $K_4$ minus an edge). 

One important class of bipartite medianlge graphs are  \emph{median graphs}, which can be defined as the mediangle graphs in which the cycle $C$ in the Cycle Condition is always a 4-cycle. Alternatively, median graphs are defined by the existence of a unique median vertex of triplets. They  have a rich structure and their cube complexes  are contractible. Moreover, by the results of  \cite{Ch_CAT,MR1663779,Ro},  median graphs are the 1-sekeleta of CAT(0) cube complexes. Due to this bijection, they are crucial in the study of groups acting geometrically on CAT(0) cube complexes in geometric group theory. Other examples include bipartite cellular graphs \cite{BaCh_cellular}, hypercellular graphs \cite{CKM20},  and $1$-skeleta of some small cancellation polygonal complexes (see \cite{Ge_mediangle} for details).  Another important class of bipartite mediangle graphs are~\emph{Coxeter graphs}, i.e., Cayley graphs of Coxeter 
systems. Also for them a 
cell structure is inherent:  finite Coxeter subgroups naturally define a higher-dimensional cell structure on Coxeter graphs, which can be endowed with a CAT(0) geometry \cite[Chapter 12]{Davis}.  Quasi-median graphs \cite{BaMuWi} are the graphs in which uniqueness of medians is replaced by uniqueness of quasi-medians, which are products of vertices and triangles. While median graphs isometrically embed into hypercubes, quasi-median graphs are not bipartite but isometrically embed into Hamming graphs.   Similarly to previous examples, quasi-median graphs lead to contractible cell complexes  \cite{QM}. 

\subsection{Main results} 
The above leads naturally to the  question~\cite[Question 7.1]{Ge_mediangle} whether mediangle graphs can be endowed with a structure of a contractible cell complex.  In this article, we solve this problem. More precisely, we prove the following result:

\begin{theorem} {\sf (Mediangle graphs have contractible cell complexes)}\label{thm:GeneralCase}
A mediangle graph $G$ is the $1$-skeleton of a contractible cell complex $X(G)$. Moreover, the $1$-skeleta of cells in $X(G)$ coincide with the gated-antipodal subgraphs of $G$. 
\end{theorem}

Notice that, from the perspective of geometric group theory, an immediate application of Theorem~\ref{thm:GeneralCase} is:

\begin{corollary}
A group $\Gamma$ acting properly and cocompactly on a mediangle graph is of type $\mathcal{F}_\infty$. If $\Gamma$ is moreover torsion-free, then it is of type $\mathcal{F}$. 
\end{corollary}

Recall that, given an integer $n\geq 1$, a group is \emph{of type $\mathcal{F}_n$} if it admits a classifying space with finitely many cells of dimension $\leq n$. A group is of type $\mathcal{F}_1$ (respectively, $\mathcal{F}_2$) if and only if it is finitely generated (respectively, finitely presented). A group is \emph{of type $\mathcal{F}_\infty$} if it is of type $\mathcal{F}_n$ for every $n \geq 1$. This amounts to saying that the group admits a classifying space with finitely many cells in each dimension (but possibly infinitely many in total). Finally, a group is \emph{of type $\mathcal{F}$} if it admits a classifying space with finitely many cells. Every group acting properly and cocompactly on a contractible cell complex is automatically of type $\mathcal{F}_\infty$. But it may not be of type $\mathcal{F}$. For instance, (non-trivial) finite groups are of type $\mathcal{F}_\infty$ but never of type $\mathcal{F}$. But a torsion-free group acting properly and cocompactly on a contractible cell complex is automatically of type $\mathcal{F}$. 

As we noted above,  most of the major classes of known examples of mediangle graphs are bipartite.  The bipartite case is also an important step in the proof of Theorem \ref{thm:GeneralCase}.  In the bipartite case, the cell complex $X(G)$ has a precise combinatorial and geometric interpretation: 


\begin{theorem} {\sf (Bipartite mediangle graphs are COMs)}
\label{t:mediangle->COM} Any bipartite mediangle graph $G$ is the tope graph of a finitary Complex of Oriented Matroids and thus can be endowed with the structure of contractible cell complex.
\end{theorem}

Finite Complexes of Oriented Matroids (COMs) were introduced by Bandelt, Chepoi, and Knauer~\cite{BaChKn}. The cells of these complexes are Oriented Matroids (OMs) and those complexes are contractible cell complexes~\cite{BaChKn}. Oriented Matroids are combinatorial objects co-invented by Bland and Las Vergnas \cite{BlLV} and Folkman and Lawrence \cite{FoLa}, and further investigated by many  authors (see the books \cite{BjLVStWhZi,and25} or the recent survey~\cite{ziegler2024oriented}). OMs capture combinatorics of sign vectors representing the regions in a central 
hyperplane arrangement in ${\mathbb R}^d$, but are more general than this \emph{realizable} case. 
A useful tool for studying  hyperplane arrangements is the incidence graphs of their maximal cells (which in case of central arrangements correspond to $1$-sekeleta of zonotopes, see~\cite{Zi}). The analogous construction for 
OMs are tope graphs, which faithfully represent the whole object. They were introduced by Edmonds and Mandel 
\cite{ed-ma-82}, but already played an important role in Tits' \cite{Ti}
combinatorial study of reflection arrangements (Coxeter complexes). Further, tope graphs were used by Deligne \cite{De} 
in the study of simplicial arrangements, later investigated by  Bj\"orner, Edelmann, and Ziegler~\cite{BjEdZi}, whose OM generalization is crucial in our study. Deligne proved that the distance 
between two vertices in the  tope graph of a hyperplane arrangement is equal to the number of hyperplanes of the arrangement  separating the corresponding regions. This result corresponds to the more general fact that the tope graph of an OM can be isometrically embedded into a hypercube, i.e., such a graph is a \emph{partial cube}, see \cite{BjLVStWhZi}. 

Analogously to OMs, the tope graphs of COMs are  partial cubes and   determine the full complex up to isomorphism~\cite{BaChKn}. 
Knauer and Marc \cite{KnMa} characterized the tope graphs of COMs as those partial cubes all whose antipodal subgraphs are gated. The 1-skeleta of cells of COMs are the antipodal 
subgraphs and OMs are exactly those COMs whose tope graph is antipodal.

Our Theorem~\ref{thm:GeneralCase} characterizes the 1-skeleta of cells of mediangle graphs $G$ as the gated-antipodal subgraphs of $G$. On the other hand, by the results of \cite{BaChKn,KnMa}, the 1-skeleta of cells of COMs are the gated-antipodal subgraphs of their tope graphs and they are OMs. Thus, it is natural to ask 
for a more precise description of the cells that appear in cell complexes of mediangle graphs. In other words, we need to describe more carefully the structure of gated-antipodal mediangle graphs. 
While the cells in COMs are as general as OMs, we show that the cells of the COMs arising from biparite mediangle graphs are much more specific. The following result fully describes the structure of cells in the cell complexes of mediangle graphs.

\begin{theorem}  {\sf (Cells of mediangle graphs)}\label{thm:MediangleGatedAntipod}
A mediangle graph $G$ is gated-antipodal (i.e., $G$ is the $1$-skeleton of a cell of the cell complex of a mediangle graph) if and only if $G$ is the Cartesian product of finitely many complete graphs and a graph $H$ satisfying any of the following equivalent conditions:
    \begin{itemize}
        \item[(i)] $H$ is an antipodal, bipartite mediangle graph;
        \item[(ii)] $H$ is an antipodal, apiculate  partial cube;
        \item[(iii)] $H$ is the tope graph of a simplical OM.
    \end{itemize}
\end{theorem}

The equivalence (ii)$\Leftrightarrow$(iii) in Theorem \ref{thm:MediangleGatedAntipod} was conjectured by the third author in \cite[Conjecture 10]{Knauer_HDR} and creates a link to order theory. Namely, for a graph $G=(V,E)$ and a basepoint $u\in V$, we denote by $G_{\preceq_u}$ the poset on $V$ where $x\preceq_u y$ if and only if 
$x\in [u,y]$. 
A graph $G$ is \emph{apiculate} \cite{BaCh_wma1} if $G_{\preceq_u}$ is a meet semilattice for any $u$. The additional assumption of antipodality corresponds to $G_{\preceq_u}$ being a lattice for every choice of basepoint $u$. 
\emph{Simplicial oriented matroids} are generalizations of simplicial hyperplane arrangements \cite{BjEdZi,BjLVStWhZi,De}, i.e., arrangements of hyperplanes in a real vector space ${\mathbb R}^d$ in which all regions are simplicial cones. By \cite[Proposition 2.3.6]{BjEdZi}, finite reflection arrangements are simplicial (see Section~\ref{section:Matroids} for more information on oriented matroids).  

A second stronger part of~\cite[Question 7.1]{Ge_mediangle} is whether mediangle graphs admit a CAT(0) metric, i.e., if the associated cell complex is CAT(0). However,
not all simplicial OMs are realizable by a Euclidean hyperplane arrangement, i.e., their tope graphs are not graphs of zonotopes. Hence by Theorem \ref{thm:MediangleGatedAntipod} it is not clear how to make sense of this part of the question. This and other perspectives will be addressed in Section~\ref{sec:fin}.

\subsection{A few words about the proofs}  
We conclude this introductory section with a  few words about the proofs of our results. The first step towards the proof of Theorem~\ref{thm:GeneralCase} is to understand the bipartite case (Theorem~\ref{t:mediangle->COM}). We prove Theorem~\ref{t:mediangle->COM} in two steps: first we prove that any bipartite mediangle graph $G$ is apiculate (Lemma \ref{mediangle->apiculate}) and then we prove that any apiculate partial cube $G$ is a tope graph of a COM (Lemma \ref{l:apiculate->COM}). While the theory of OMs and COMs was developed in the finite case, the mediangle graphs considered in our results may be infinite. Therefore, to prove the contractibility of their complexes (via Whitehead's theorem), we have to extend the theory of finite COMs to finitary COMs (which we introduce for this purpose in this paper and which extend finitary affine oriented matroids of \cite{DK25}).  The proof of Theorem \ref{thm:MediangleGatedAntipod} in the bipartite case  uses Lemma \ref{mediangle->apiculate} and   some deep results from the theory of COMs and OMs.

The second step towards the proofs of Theorems~\ref{thm:GeneralCase} and \ref{thm:MediangleGatedAntipod} is to understand how the global structure of a non-bipartite mediangle graph can be reduced to some of its  bipartite mediangle subgraphs.  It is well-known that the geometry of a median graph is essentially encoded into the combinatorics of its \emph{hyperplanes} (or \emph{$\Theta$-classes}). Similarly, hyperplanes can be defined in mediangle graphs, providing a very useful tool in the study of the geometry of mediangle graphs. Contrary to median graphs, hyperplanes in mediangle graphs can cross in different ways. More precisely, any pair of transverse hyperplanes can be assigned an \emph{angle}.  In median graphs, the angle between two transverse hyperplanes is $\pi/2$. This motivates the introduction of a specific family of hyperplanes, dubbed \emph{right hyperplanes}. Formally, a hyperplane $J$ is \emph{right} if $\measuredangle (J,H)=\pi/2$ for every hyperplane $H$ transverse to $J$. The idea to keep in mind is that, in a mediangle graph, a right hyperplane should behave ``like a hyperplane in a median graph''. In this perspective, our main result is that sectors delimited by right hyperplanes are always gated (Proposition~\ref{prop:RightHypGated}).  It turns out that a mediangle graph that is not bipartite always contains a right hyperplane. Then, such a hyperplane can be used  to decompose the mediangle graph as a gated amalgamation  (Proposition~\ref{prop:RightHypGated}). In a nutshell, such decompositions allow us to reduce the proofs of Theorem~\ref{thm:GeneralCase} and of the first part of Theorem \ref{thm:MediangleGatedAntipod} to the bipartite case. 


\section{Preliminaries}

In this section, we record basic definitions and properties first regarding abstract graphs, in Section~\ref{section:Graphs}, and then regarding mediangle graphs, in Section~\ref{section:Mediangle}.

\subsection{Graphs}\label{section:Graphs}

All graphs $G=(V,E)$ occurring in this paper are simple, connected, without loops or multiple edges, but not necessarily finite. We will write $u\sim v$ if two vertices $u$ and $v$ are adjacent and $u\nsim v$ if $u$ and $v$ are not adjacent. 
By a \emph{clique} of $G$ we will mean an inclusion maximal complete subgraph $K$ of $G$. The {\it distance} $d(u,v):=d_G(u,v)$ between two vertices $u$ and $v$ is the length
of a shortest $(u,v)$-path, and the {\it interval} $[u,v]$
between $u$ and $v$ consists of all vertices on shortest
$(u,v)$--paths:
$$[u,v]:=\{ x\in V: d(u,x)+d(x,v)=d(u,v)\}.$$
We will use in many places, without refereeing,  the following standard lemma about intervals in graphs: 

\begin{lemma}\label{l:between} 
If $x\in [u,v]$ and $y\in [x,v]$, then $y\in [u,v]$ and $x\in [u,y]$. \qed
\end{lemma}

An induced subgraph $H$ of $G$ (or the corresponding vertex set $A$)
is called {\it convex} if it includes the interval of $G$ between
any two of its vertices. A \emph{halfspace} is a convex set $A$ with convex complement $V\setminus A$. An induced subgraph $H$ of $G$ is {\it
isometric} if the distance between any pair of vertices in $H$ is
the same as that in $G$.  In particular, convex subgraphs are
isometric.  The \emph{imprint} of a vertex $u$ in a set $A\subseteq V$ is the set $\Imp_A(u)=\{ x\in A: [u,x]\cap A=\{ x\}\}$. 
If $|\Imp_A(u)|=1$ for any $u\in V\setminus A$, then the set $A$ (or the subgraph $G[A]$) is called \emph{gated}~\cite{DrSch} and the unique vertex  of $\Imp_A(u)$ is called the \emph{gate} of $u$ in $A$ and is denoted by  $\gate_A(u)$. Equivalently, $A$ is gated if for any $u\in V\setminus A$ there exists a unique vertex $\gate_A(u)\in A$ such that $\gate_A(u)\in [u,v]$ for any $v\in A$. Gated sets are convex, but the converse does not hold (e.g.\ edges in complete graphs). Nevertheless, gated sets are closed by taking intersections, therefore, for each set 
$A\subseteq V$ there exists the smallest gated set $\lgate A\rgate$ containing $A$ and it is called the \emph{gated hull} of $A$.  

An \emph{antipode} of a vertex $v$ in an induced  subgraph $H=G[A]$ is a (necessarily unique) vertex $\ov$ such that $A=[v, \ov]$, where $[v,\ov]$ is considered in $G$. A subgraph $H\subseteq G$ is \emph{antipodal} if all its vertices have antipodes. For mediangle graphs we will introduce the notion of \emph{gated-antipodal} subgraphs, see \Cref{def:gatan}.



A graph $G=(V,E)$ is
{\it isometrically embeddable} into a graph $H=(V',E')$ if there
exists a mapping $\varphi : V\rightarrow V'$ such that $d_H(\varphi
(u),\varphi (v))=d_G(u,v)$ for all vertices $u,v\in V$,
 i.e., $\varphi(G)$ is  an isometric subgraph of $H$. Given a (possibly infinite) set $U$, the \emph{hypercube} $Q(U)$ has the finite subsets of $U$ as the vertex set and two vertices are adjacent if they differ in exactly one element. The distance between two vertices $A$ and $B$ of $Q(U)$ is the size of the symmetric difference $A\Delta B=(A\setminus B)\cup (B\setminus A)$. 
 A \emph{face} of $Q(U)$ is the subgraph $Q$  of $Q(U)$ induced by all sets $C$ such that $A\subseteq C\subseteq B$ for two finite subsets $A\subseteq B\subset U$. Then the face is a $k$-dimensional hypercube, where $k$ is the size of $B\setminus A$. 
 A hypercube can also be defined as the weak Cartesian product of $K_2$. More generally, a \emph{Hamming graph} is a weak Cartesian product $\Pi_{i\in U} K_{a_i}$ of complete graphs $K_{a_i}, i\in U$. Analogously to hypercubes, the distance between two vertices $x,y$ of a Hamming graph is the Hamming distance between the tuples corresponding to the vertices $x,y$. A graph $G$ is called a {\it partial cube} (respectively, a \emph{partial Hamming graph}) if it admits an isometric embedding into some hypercube (respectively, into some Hamming graph). 
From now on, we will always suppose that a partial cube or partial Hamming graph $G=(V,E)$ is an isometric subgraph of the hypercube $Q(U)$ (i.e., we will identify $G$ with its image under the isometric embedding). Although we do not use it, notice that partial cubes have been nicely characterized by Djokovi\'{c}~\cite{Dj} in terms of convexity. Djokovi\'{c}'s characterization was extended to all partial Hamming graphs in \cite{Ch_Hamming}; see also  \cite{Wi} for other characterizations of partial Hamming graphs. 

For a triple of vertices $u,x,y$ of a graph $G$, a $u$-{\it apex}
relative to $x$ and $y$ is a vertex $u'\in
[u,x]\cap [u,y]$ such that $[u,u']$ is maximal with respect to inclusion. A graph $G$ is
{\it apiculate}~\cite{BaCh_wma1} if for any vertex $u$ and any pair of vertices $x,y$ of $G$, the $u$-apex relative to $x$ and $y$ is unique. Equivalently, $G$ is apiculate if and only if  for any vertex $u$ the vertex set of $G$
is a meet semilattice with respect to the \emph{base-point order}
$\preceq_u$ defined by setting $v\preceq_u v'$ if and only if $v\in [u,v']$. 
The meet of $x,y$ in this semilattice is exactly the $u$-apex. We denote the semilattice  by  $G_{\preceq_u}$. For a set $A\subseteq V$ we denote by $A_{G_{\preceq_u}}$ the restriction of $\preceq_u$ to $A$.

\subsection{Mediangle graphs}\label{section:Mediangle}

In this subsection we recall  the formal definition of mediangle graphs and their basic properties from \cite{Ge_mediangle}.

\begin{definition}[Mediangle graphs~\cite{Ge_mediangle}]  A graph $G=(V,E)$ is \emph{mediangle} if it satisfies the following conditions:
\begin{description}
	\item[(Triangle Condition)] For all vertices $u,x,y\in V$ satisfying $d(u,x)=d(u,y)$ and $d(x,y)=1$, there exists a common neighbor $w$ of $x$ and $y$ that belongs to $[u,x] \cap [u,y]$.
    	\item[(Cycle Condition)] For all vertices $u,x,y,z\in V$ satisfying 
        $d(u,x)=d(u,y)=d(u,z)-1$ and $d(x,z)=d(y,z)=1$, there exists an even 
        convex cycle $C$ that contains the edges $zx,zy$ and the vertex opposite to $z$ in $C$ belongs to $[u,x]\cap [u, y]$.
	\item[(Intersection of Even Cycles Condition)] The intersection of any two even convex  cycles contains at most one edge.
    \item[(Diamond Condition)] $G$ does not contain an induced subgraph isomorphic to the diamond $K_4^-$ (i.e.\ the complete graph $K_4$ minus an edge, or equivalently two triangles glued along an edge). 
\end{description} 
\end{definition}

\begin{center}
\begin{figure}
    \centering
    \includegraphics[width=0.6\linewidth]{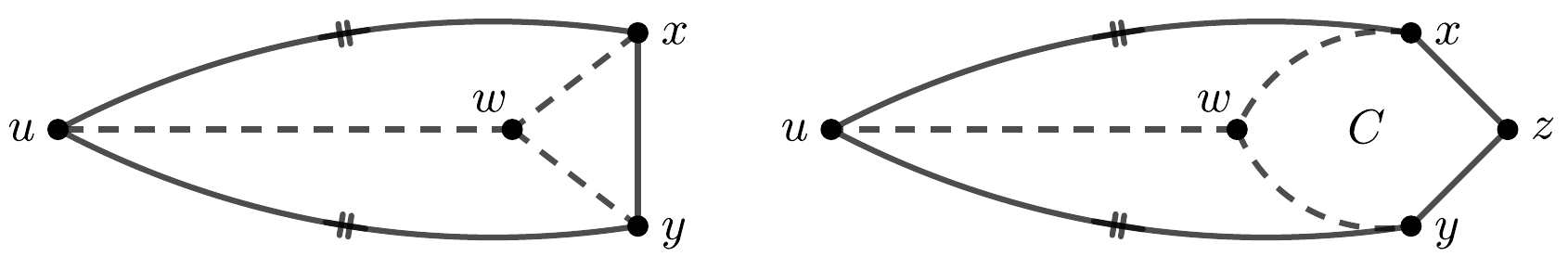}
\caption{Triangle and Cycle Conditions} 
    \label{fig:placeholder}
\end{figure}
\end{center}



It is clear from the definition of mediangle graphs that even convex cycles play a central role. An important property is that in almost all cases they are  gated:

\begin{theorem}[\hspace{-.5pt}{\cite[Theorem 3.5]{Ge_mediangle}}] \label{convex-cycles->gated}  
In a mediangle graph $G$, convex cycles of even length $>4$ are gated and cycles of length 4 are convex. Consequently, if $G$ is bipartite, each convex cycle is gated. 
\end{theorem}  

In the bipartite case, mediangle graphs are partial cubes, (see Theorem~\ref{thm:BigHyp}(4) below), 
and in partial cubes all cycles of length 4 are gated. Hence, the second part of the statement follows from the first part. Notice that in an arbitrary mediangle graph, convex cycles of length $4$ may not be gated (e.g.\ in $K_2 \square K_3$).

In a mediangle graph $G$, a \emph{hyperplane} is an equivalence class of edges with respect to the reflexive-transitive closure of the relation that identifies two edges whenever they belong to a common clique or are opposite edges in an even convex cycle. As illustrated  by the next statement, the geometry of a mediangle graph is essentially encoded into the combinatorics of its hyperplanes.

\begin{theorem}[{\hspace{1sp}\cite[Theorem~3.9 and Theorem~1.5]{Ge_mediangle}}]\label{thm:BigHyp}
Let $G$ be a mediangle graph.
\begin{enumerate}
     
	\item A hyperplane $J$ separates $G$, i.e.\ the graph $G\backslash J$ 
    obtained from $G$ by removing all the edges from $J$ contains at least two connected components, called the \emph{sectors} delimited by $J$.
	\item Sectors delimited by hyperplanes are convex.
	\item A path is a geodesic if and only if it crosses each hyperplane at most once. As a consequence, the distance between any two vertices coincides with the number of hyperplanes separating them. 
    \item {Mediangle graphs are clique-gated partial Hamming graphs and bipartite mediangle graphs are partial cubes.}
\end{enumerate}
\end{theorem}

From the definition of mediangle graphs, it immediately follows that convex subgraphs of mediangle graphs are mediangle. Since intervals in partial Hamming graphs are convex, intervals in  mediangle graphs are also mediangle: 

\begin{lemma}\label{lem:convex} Convex subgraphs, in particular intervals, of  mediangle graphs are mediangle. \qed
\end{lemma}

\section{Bipartite mediangle graphs and COMs}

This section is dedicated to the proofs of Theorem~\ref{t:mediangle->COM} and the second part of Theorem \ref{thm:MediangleGatedAntipod}. 
This will be our first step towards the proof of Theorems~\ref{thm:GeneralCase} and \ref{thm:MediangleGatedAntipod}. 
Some definitions and results about OMs and COMs are recorded in Subsection~\ref{section:Matroids}. Theorem~\ref{t:mediangle->COM} is proved in Subsection~\ref{section:ProofBipartite} and the second part of Theorem \ref{thm:MediangleGatedAntipod} is proved in  Subsection~\ref{section:ProofBipartiteCell}.

\subsection{OMs and COMs}\label{section:Matroids}

We recall the basic theory of finite OMs
and COMs from~\cite{BjLVStWhZi} and~\cite{BaChKn}, respectively. 
Then we define finitary COMs using the definition of \cite{DK25}.

Let $U$ be a (possibly infinite) set called \emph{universe} or \emph{ground set} and let $\covectors$ be a {\it system of
sign vectors}, i.e., maps from $U$ to $\{-1,0,+1\}$. 
A system $\covectors$ of sign-vectors is \emph{simple} if $\{X_e: X\in \covectors\}=\{-1,0,+1\}$ for all $e\in U$. We will assume this property to hold for the entire article. The elements of
$\covectors$ are called \emph{covectors} and denoted by capital letters
$X, Y, Z$. For $X \in \covectors$, the subset $\underline{X} = \{e\in U:
X_e\neq 0\}$ is the \emph{support} of $X$ and  its complement
$X^0=U\setminus \underline{X}=\{e\in U: X_e=0\}$ is the \emph{zero set} of $X$. Let also $X^-=\{ e\in U: X_e=-1\}$ and $X^+=\{ e\in U: X_e=+1\}$. For $X,Y\in \covectors$, $\Sep(X,Y)=\{e\in U: X_eY_e=-1\}$ is the
\emph{separator} of $X$ and $Y$. The \emph{composition} of $X$ and $Y$ is the
sign
vector $X\circ Y$, where for all $e\in U$,
$(X\circ Y)_e = X_e$ if $X_e\ne 0$  and  $(X\circ Y)_e=Y_e$ if $X_e=0$. 

\begin{definition} [Complex of Oriented Matroids] \label{def:COM}
A \emph{complex of oriented matroids} (COMs) is a system of sign vectors
$\M=(U,\covectors)$ on a finite ground set $U$ satisfying the following axioms:
\begin{itemize}
		\item[{\bf (SE)}] ({\sf Strong elimination}) for each pair
		$X,Y\in\covectors$ and for each $e\in \Sep(X,Y)$, there exists $Z \in
		\covectors$ such that $Z_e=0$ and $Z_f=(X\circ Y)_f$ for all $f\in
		U\setminus \Sep(X,Y)$.
	\item[{\bf (FS)}] ({\sf Face symmetry}) $X\circ -Y \in  \covectors$
	for all $X,Y \in  \covectors$.
\end{itemize}
\end{definition}

\begin{figure}[ht]
    \centering
    \includegraphics[width=.9\linewidth]{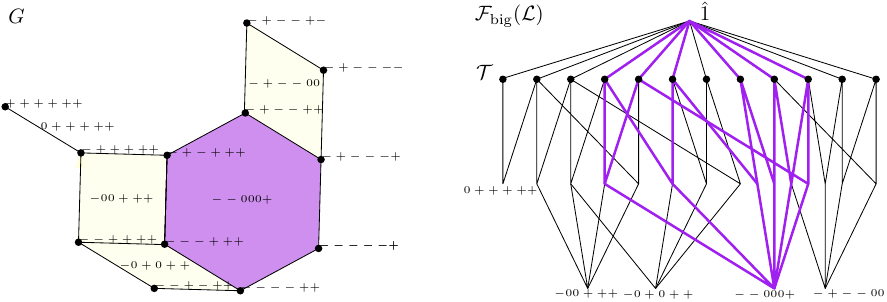}
    \caption{Left: a bipartite mediangle graph $G$ with 4 antipodal subgraphs of rank 2. Right: the big face semilattice of the COM $\mathcal{M}$ whose tope graph is $G$. The purple antipodal subgraph on the left corresponds to a face on the right. In order to not convolute the drawing only  minimal and maximal covectors are labelled.}
    \label{fig:medianlge}
\end{figure}

The axioms {\bf (SE)} and {\bf (FS)} imply that $\covectors$ is closed under composition, i.e., if $X,Y\in \covectors$, then $X\circ Y\in \covectors$. 
A very important subclass of COMs are OMs. The following definition of OMs is more economic but equivalent to others~\cite{BaChKn}:
\begin{definition}[Oriented Matroid] \label{def:OM}
	An \emph{oriented matroid} (OM)  is a COM
$\M=(U,\covectors)$ on a finite ground set $U$ that  additionally satisfies
	\begin{itemize}
        \item[{\bf (Z)}]  the zero sign vector ${\bf 0}$ belongs to $\covectors$.
	\end{itemize}
\end{definition}

Let $\leq$ be the product ordering on $\{-1,0,+1\}^{U} $
relative to  the ordering $0 \leq -1, +1$.
The poset $\mathcal{F}_{\mathrm{big}}(\covectors)=(\covectors\cup\{\hat{1}\},\le)$ of a COM $\cM$ with an artificial global maximum
$\hat{1}$ forms the (graded)
\emph{big face semilattice}. The length of
a longest chain in $\mathcal{F}_{\mathrm{big}}(\covectors)$ minus $\hat{1}$
is the \emph{rank} of $\cM$ and denoted $\rk(\cM)$.  
The \emph{topes} $\topes$ of $\cM$ are the co-atoms of
$\mathcal{F}_{\mathrm{big}}(\covectors)$, i.e., those elements $T\in \covectors$ such that in $\mathcal{F}_{\mathrm{big}}(\covectors)$ only $T$ itself and $\hat{1}$ are larger or equal $T$. 
By assuming simplicity of $\cM$ the topes 
are  $\{-1,+1\}$-vectors. Thus, ${\mathcal T}$ can be seen as a family of subsets
 of $U$, where for each $T\in\mathcal{T}$, an element
 $e\in U$ belongs to the corresponding
 set if and only if $T_e=+1$. 
The \emph{tope graph} $G(\cM)$ of a COM $\cM$ is 
the subgraph of the hypercube $\{+1,-1\}^U$ induced by the vertices corresponding to
$\topes$. Tope
graphs of COMs are partial cubes and $\M$ can be recovered up to
isomorphism from $G(\M)$. {Thus, COMs can be treated as graphs.} 

\begin{definition} [Simpliciality] \label{def:SOM}
A tope $T$ of an OM $\cM=(U,\covectors)$ of rank $r$ is \emph{simplicial} if the degree of $T$ in the tope graph $G(\cM)$ is $r$, see~\cite[pp.161 \& 185]{BjLVStWhZi}. Equivalently, $T$ is simplicial if the order interval $[\varnothing,T]$ in $\mathcal{F}_{\mathrm{big}}(\covectors)$ is a Boolean lattice. An OM is \emph{simplicial} if all its topes are simplicial, see~\cite[pp.161]{BjLVStWhZi}. 
\end{definition}

For a covector $X\in \covectors$, the
\emph{face} of $X$ is $\Face(X):=\{Y\in\covectors: X\leq Y\}$ and the \emph{star} of $X$ is the set of topes 
 $\Star(X):=\{Y\in{\mathcal T}: X\leq Y\}$, see~\cite{BaChKn,BjLVStWhZi}. We will denote by $\Cell(X)$ the subgraph of the tope graph $G(\cM)$ induced by $\Star(X)$ and call it the \emph{cell of $X$}. A \emph{facet} of  $\Face(X)$ is a maximal face $\Face(Y)$ of $\cM$ properly contained in $\Face(Y)$. 
By \cite[Lemma 4]{BaChKn}, each face $\Face(X)$ of a COM $\cM$ is isomorphic to an OM. The topes in $\Star(X)$ induce a subgraph $\Cell(X)$ of $G(\covectors)$, which is isomorphic
to the tope graph $G(\Face(X))$ of $\Face(X)$. The COMs and the cells of the tope graph of COMs have been  characterized in the following way:

\begin{theorem}[\hspace{-.5pt}\cite{KnMa}] \label{cells=gated+antipodal}  A partial cube $G$ is the tope graph of a COM if and only if all antipodal subgraphs of $G$ are gated. Furthermore, a subgraph $G'$ of a tope graph of a COM $\cM=(U,\covectors)$ is a cell of $G(\cM)$ (i.e., there exists $X\in \covectors$ such that $G'=\Cell(X)$) if and only if $G'$ is an antipodal (and thus gated) subgraph of $G(\cM)$. 
\end{theorem}

COMs lead to contractible regular cell complexes.  By \cite[Section 11]{BaChKn},  replacing each combinatorial face $\Face(X)$ of a COM $\cM$ by a PL-ball $\sigma(\Face(X))$, we obtain a regular cell complex $\Delta(\covectors)$. 

\begin{theorem}[\hspace{-.5pt}{\cite[Proposition 15]{BaChKn}}] \label{COMs-contractible}   For a COM $\cM=(U,\covectors)$, 
the cell complex $\Delta(\covectors)$ is contractible.
\end{theorem}

To define infinite COMs as systems of sign vectors of $\{ \pm 1,0\}^U$, we are lead to 
the notion of  finitary COMs introduced in  \cite{DK25} in the study of finitary affine oriented matroids:
\begin{definition}[Finitary COM]
 A \emph{Finitary} Complex of Oriented Matroids is a 
 system of sign vectors
	$\M=(U,\covectors)$ on a countable set $U$ satisfying (SE)  and (FS) and the following axioms:
 \begin{itemize}
  \item[{\bf (S)}] $X,Y\in\covectors\implies |\Sep (X,Y)|<\infty$ \hfill(finite separators),
  \item[{\bf (PZ)}] $X\in\covectors\implies |X^+|,|X^0|<\infty$ \hfill(finite positive and zero sets),
  \item[{\bf (I)}] $\vert \mathcal{F}_{\mathrm{big}}(\covectors)_{\leq X}\vert < \infty$ \hfill(finite ideals).
 \end{itemize}
\end{definition}

From the definition it follows that finite restrictions of finitary COMs are finite COMs. In particular, $|X^0|<\infty$ yields that faces are finite OMs. The assumption $|X^+|<\infty$ for all $X\in \covectors$, allows to view $\mathcal{T}$ as a subset of vertices of the hypercube $Q(U)$, which by definition consists of the finite subsets of $U$. In particular, we get that the tope graph $G(\cM)$ of a finitary COM is an isometric subgraph of $Q(U)$.
Theorem \ref{COMs-contractible} extends to finitary COMs (for which the cell complex $\Delta(\covectors)$ is defined as in the finite case):

\begin{theorem}\label{finitaryCOMs-contractible}
    The cell complex $\Delta(\covectors)$ of a finitary COM $\cM=(U,\covectors)$ is contractible.
\end{theorem}
\begin{proof}
Consider any total ordering of the set $U=\{ e_1,e_2,\ldots,\}$. For any $i$, let $U_{i}=\{ e_1,\ldots,e_i\}$, $\cM_i=\cM_{|U_i}$ be the finite COM obtained by restricting to $U_i$, and let $\Delta(\covectors_i)$ be the cell complex of $\cM_i$. Then  $\Delta(\covectors_i)$ is contractible by Theorem \ref{COMs-contractible}. The cell complex $\Delta(\covectors)$ is the directed union of the contractible finite cell complexes $\Delta(\covectors_i)$. Thus, $\Delta(\covectors)$ is contractible by the classical theorem of Whitehead \cite[Theorem 4.5]{Hat}. 
\end{proof}

\subsection{Proof of Theorem~\ref{t:mediangle->COM}}\label{section:ProofBipartite}

To prove  that bipartite mediangle graphs are tope graphs of COMs, we show that bipartite mediangle graphs are apiculate and then that apiculate partial cubes are tope graphs of COMs.

\begin{lemma}\label{mediangle->apiculate}
    If $G=(V,E)$ is a bipartite mediangle graph, then $G$ is apiculate.
\end{lemma}
\begin{proof} Pick any base-point $u$ and any two vertices $x,y$.  To prove that the $u$-apex relative to $x$ and $y$ is unique, we proceed by induction on  $d(u,x)+d(u,y)$. Suppose by way of contradiction that there exist  two distinct $u$-apices $a$ and $b$ relative to $x$ and $y$.  Then $a,b\ne u$. First, suppose that  $[u,a]\cap [u,b]\ne \{u\}$. Then there exists a neighbor $u'$ of $u$ such that $u'\in [u,a]\cap [u,b]$. Then $a,b\in [u',x]\cap [u',y]$. Since $d(u',x)+d(u',y)<d(u,x)+d(u,y)$, by the induction assumption, $x$ and $y$ admit a unique $u'$-apex, denote it by $c$. Since $a,b\in [u',x]\cap [u',y]$, by the definition of the apex, we have $a,b\in [u',c]$. Since $u'\in [u,x]\cap [u,y]$ and $c\in [u',x]\cap [u',y]$, by Lemma  \ref{l:between} we get that $c$ is  a $u$-apex with respect to $x$ and $y$. This contradicts the assumption that $a$ and $b$ are $u$-apices of $x$ and $y$. Consequently,  $[u,a]\cap [u,b]=\{u\}$.

Let $x'$ be a neighbor of $u$ in $[u,a]\subseteq [u,x]$ and $y'$ be the neighbor of $u$ in $[u,b]\subseteq [u,y]$. Since  $[u,a]\cap [u,b]=\{u\}$,  $x'$ and $y'$ are different and $x'\notin [u,b], y'\notin [u,a]$. Since $G$ is bipartite mediangle, by Theorem \ref{convex-cycles->gated} of \cite{Ge_mediangle}, there exists a gated cycle $C_x\subseteq [u,x]$ containing the vertices $u,x',y'$. Analogously, there exists a gated cycle $C_y\subseteq [u,y]$ containing the vertices $u,x',y'$. The intersection of $C_x$ and $C_y$ contains the 2-path $x'uy'$. Since the intersection of two gated sets is gated, $C_x\cap C_y$ must be gated.  Since no proper subpath of length at least 2 of a gated cycle is gated, we deduce that $C_x=C_y$. Denote this gated cycle by $C$ and let $\ou$ be the vertex of $C$ opposite to $u$. Denote also by $a'$ and $b'$ the gates of the vertices $a$ and $b$ in $C$. By the definition of bipartite mediangle graphs and since $C=C_x=C_y$, $\ou$ belongs to both intervals $[u,x]$ and $[u,y]$. Since $\ou$ is the antipode of $u$ in $C$ and $C$ is gated, this implies that $\ou$ is the gate of $x$  and $y$ in $C$. Since $a,b\in [u,x]\cap [u,y]$ and $a'\in [u,a], b'\in [u,b]$,  by Lemma  \ref{l:between} we get $a',b'\in [u,x]\cap [u,y]$. 

If $a'=\ou=b'$, then  $C\subseteq [u,a]\cap [u,b]$, contrary to the assumption that  $[u,a]\cap [u,b]=\{u\}$. So, suppose that $a'\ne \ou$. Since $y'\notin [u,a]$,  necessarily $a'$ belongs to the subpath of $C$ between $u$ and $\ou$  passing via $x'$. Since $x'\in [u,a]$, $a'$ is different from $u$. Therefore the distance sum $d(a',x)+d(a',y)$ is smaller than $d(u,x)+d(u,y)$. By the induction hypothesis, there exists a unique $a'$-apex $t$ relative to $x$ and $y$. Since $a,\ou\in [a',x]\cap [a',y]$ (that $\ou\in [a',x]\cap [a',y]$ follows from the conclusion that $\ou$ is the gate of $x$ and $y$ in $C$), necessarily $a,\ou\in [a',t]$. Since $a'\in [u,x]\cap [u,y]$, Lemma \ref{l:between} implies that $t\in [u,x]\cap [u,y]$ and $a\in [t,u]$. 
Since $a$ is an $u$-apex relative to $x$ and $y$, all this implies that $a=t$. Since $t\in [x,\ou]$ and $\ou$ is the gate of $x$ in $C$, $\ou$ is also the gate of $t=a$ in $C$. Hence $a'=\ou$, contrary to our assumption that $a'\ne \ou$. This concludes the proof. 
\end{proof}

\begin{lemma}\label{l:apiculate->COM}
If $G$ is an apiculate partial cube, then $G$ is the tope graph of a COM.
\end{lemma}
\begin{proof} We will use the characterization of COMs provided by Theorem \ref{cells=gated+antipodal} of\cite{KnMa}. Note that in partial cubes intervals are convex. Hence, antipodal subgraphs are convex. So, let $G$ be an apiculate partial cube and let $C$ be an antipodal subgraph of $G$. Suppose by way of contradiction that $C$ is not gated. Then there exists a vertex $u$ such that $|\Imp_C(u)|>1$. For a vertex $a\in \Imp_C(u)$, let $L_a=\{ y\in C: [u,y]\cap \Imp_C(u)=\{ a\}\}$. Each set $L_a$ is nonempty (because $a$ belongs to $L_a$), connected (since $x\in L_a$ and $x'\in [a,x]\subseteq C$ implies $x'\in L_a$), but is different from $C$ (since $|\Imp_C(u)|>1$). Therefore, there exists a vertex $x\in C\setminus L_a$ and $y\in L_a$ such that $x\sim y$. 
Since $x\notin L_a$, there exists a vertex $b\ne a$ such that $b\in \Imp_C(u)\cap [x,u]$. Since $y\in L_a$, we get $b\notin [y,u]$. Since the graph $G$ is bipartite,  $|d(u,x)-d(u,y)|=1$. Since $x\sim y$, all this implies that $d(u,x)=d(u,y)+1$ and thus $a\in [x,u]$. Consequently, in $C$ there exists a vertex $x$ such that $\Imp_C(u)\cap [u,x]$ contains two distinct vertices $a$ and $b$. 

Let $\ox$ denote the antipode of $x$ in $C$. Since $C=[x,\ox]$, we have $a,b\in [x,\ox]\cap [x,u]$. Consider the apex semilattice $G_{\preceq_x}$. Let $c$ be the $x$-apex with respect to $u$ and $\ox$. Since $a,b\in [x,\ox]\cap [x,u]$, we get $a,b\in [x,c]$.  Lemma \ref{l:between} implies that $c\in [u,a]\cap [u,b]$. Since $c\in [x,\ox], x,\ox\in C$, and $C$ is convex, we deduce that $c\in C$. Consequently, $c\in [u,a]\cap [u,b]\cap C$ contrary to the assumptions that $a,b\in \Imp_C(u)$ and $a\ne b$. This contradiction proves that $|\Imp_C(u)|=1$ for any vertex $u\in V\setminus C$ and thus that $C$ is gated. 
\end{proof}

Combining Lemmas \ref{mediangle->apiculate} and \ref{l:apiculate->COM} with Theorem \ref{finitaryCOMs-contractible}, we get the proof of Theorem \ref{t:mediangle->COM}. Notice that, thanks to Theorem~\ref{cells=gated+antipodal}, the bipartite case of Theorem~\ref{thm:GeneralCase} follows, namely:

\begin{corollary}\label{cor:BipartiteCase}
A bipartite mediangle graph $G$ is the $1$-skeleton of a contractible cell  complex $X(G)$. Moreover, the $1$-skeleta of cells in $X(G)$ coincide with the convex antipodal subgraphs of $G$. 
\end{corollary}

\subsection{Proof of Theorem \ref{thm:MediangleGatedAntipod} (bipartite case)}\label{section:ProofBipartiteCell}
The main goal of this subsection is to prove Theorem \ref{thm:MediangleGatedAntipod} in the bipartite case: 

\begin{theorem} \label{t:cells-in-mediangle}  
For a partial cube $G$, the following conditions are equivalent:
    \begin{itemize}
        \item[(i)] $G$ is mediangle and antipodal;
        \item[(ii)] $G$ is apiculate and antipodal;
        \item[(iii)] $G$ is the tope graph of a simplical OM.
    \end{itemize}
\end{theorem}


Recall from Section~\ref{section:Graphs} that, for a graph $G=(V,E)$ and a basepoint $u\in V$, we denote by $G_{\preceq_u}$ the poset on $V$ where $x\preceq_u y$ if and only if 
$x\in [u,y]$. A graph $G$ is \emph{apiculate} \cite{BaCh_wma1} if $G_{\preceq_u}$ is a meet semilattice for any $u$. Note that the additional assumption of antipodality corresponds to $G_{\preceq_u}$ being a lattice for every choice of basepoint $u$.

We continue with the proof of Theorem \ref{t:cells-in-mediangle}.  The implication (i)$\Rightarrow$(ii) follows from Lemma \ref{mediangle->apiculate}. 
Now, we prove (ii)$\Rightarrow$(iii). Let $G$ be apiculate and antipodal. By Lemma \ref{l:apiculate->COM} and the results of \cite{BaChKn,KnMa}, $G$ is the tope graph of an OM $\cM$. Since $G$ is apiculate, $G_{\preceq_u}$ is a meet semilattice for any basepoint $u$. Since $G$ is antipodal, $V=[u,\ou]$, the meet semilattice $G_{\preceq_u}$ has a maximal element $\ou$.  Thus $G_{\preceq_u}$ is a lattice. By \cite[Exercice 4.13]{BjLVStWhZi} (which generalizes a result of \cite{BjEdZi} for hyperplane arrangements), if $G_{\preceq_u}$  is a lattice, then the tope $u$ is simplicial. This shows that  $\cM$ is a simplicial OM, establishing (ii)$\Rightarrow$(iii). 

To prove (iii)$\Rightarrow$(i) we use Lemma 4.4.4 of \cite{BjLVStWhZi} in the particular case of simplicial OMs.  By this lemma, if $\cM$ is a simplicial OM with tope graph $G$, then for any topes $B,A,A',R$ of $\cM$ such that $A,A'\in [B,R]$ and $A,A'$ are adjacent to $B$, there exists a tope $T\in [B,R]$ such that the interval $[B,T]$ is elementary and contains $A,A'$. Recall that an interval $[T,T']$ of the tope graph $G=G(\covectors)$ is \emph{elementary} \cite[p.182]{BjLVStWhZi} if the open interval $[T,T']\setminus \{T,T'\}$ consists of two disjoint paths. Therefore the subgraph of $G$ induced by an elementary interval $[T,T']$ is a cycle. Since $[T,T']=\Star(X)$ for some $X\in \covectors$ of corank $r-2$ (see \cite[p.182]{BjLVStWhZi}) and $\Star(X)$ is gated, and thus convex, 
from Lemma 4.4.4 of \cite{BjLVStWhZi} we conclude that the tope graphs of simplicial OMs are bipartite mediangle. 
This establishes that (iii)$\Rightarrow$(i) and concludes the proof of  Theorem \ref{t:cells-in-mediangle}. 

\section{General mediangle graphs}

In this section, we complete the proof of Theorem~\ref{thm:GeneralCase}, initiated by Theorem~\ref{t:mediangle->COM}. In Subsection \ref{section:TransverseHyp}  we recall some results from \cite{Ge_mediangle} about transverse hyperplanes, the angle between transverse hyperplanes, and 
the interaction between hyperplanes and (gate-)projections on gated subgraphs  in mediangle graphs.  In Subsection~\ref{section:RightHyp}, we focus on \emph{right hyperplanes} in mediangle graphs. Roughly speaking, they will allow us to reduce the geometry of arbitrary mediangle graphs to the geometry of bipartite mediangle graphs. Then, by endowing an arbitrary mediangle graph with the structure of a cell complex whose cells are given by its gated-antipodal subgraphs, as studied in Section~\ref{section:GatedAnti}, we will be able to deduce Theorem~\ref{thm:GeneralCase} from Theorem~\ref{t:mediangle->COM} in Section~\ref{section:ProofGeneral}. 

\subsection{Hyperplanes and gated sets}\label{section:TransverseHyp} 
Given a mediangle graph $G$, the \emph{neighbourhood} $N(J)$ of a hyperplane $J$ is the subgraph of $G$ induced by $J$. A \emph{fiber delimited by $J$} is a connected component $B$ of $N(J) \backslash J$. Two hyperplanes $J_1$ and $J_2$ are \emph{transverse} if there exists an even convex cycle with two distinct pairs of opposite edges in $J_1$ and $J_2$. See Figure~\ref{fig:MediangleHyp} for an illustration.
As indicated by their names, in mediangle graphs, it is possible to define an angle by two hyperplanes whenever they are transverse.

\begin{figure}[ht]
    \centering
\includegraphics[width=0.6\linewidth]{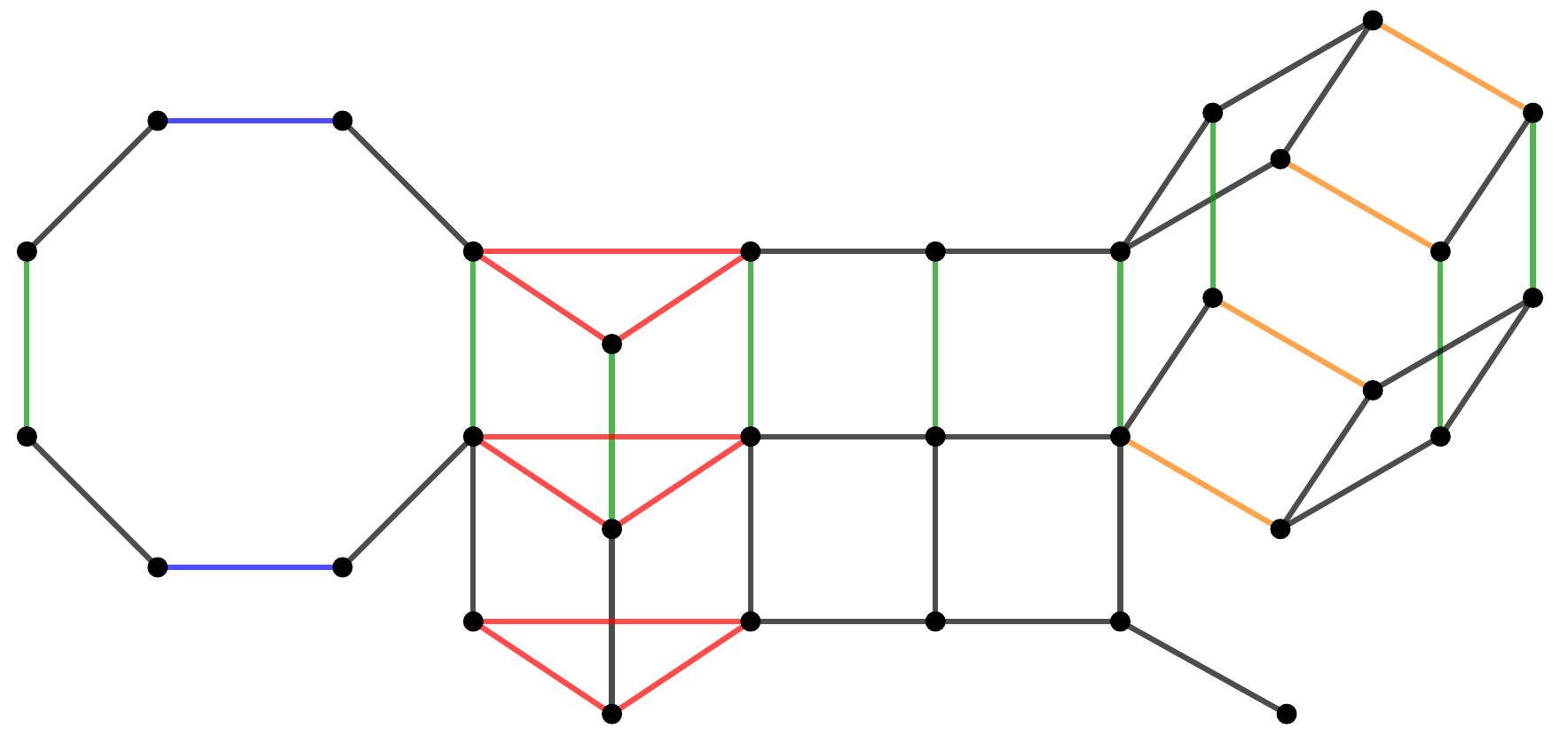}    
    \caption{If $J$ is the red hyperplane, then $N(J)$ is the disjoint union of its three fibers, each of which is a path of length two. If $J$ is the orange hyperplane, then $N(J)$ is the disjoint union of its four fibers, each of which is a path of length one. The green and the red hyperplanes are transverse, but the red and the orange are not.}
    \label{fig:MediangleHyp}
\end{figure}

\begin{definition}
In a mediangle graph, if two transverse hyperplanes $J_1$ and $J_2$ both cross an even convex cycle $C$, the \emph{angle between $J_1$ and $J_2$ at $C$} is
$$\measuredangle_C(J_1,J_2):= 2 \pi \cdot \frac{1+ d(J_1 \cap C, J_2 \cap C)}{\mathrm{length(C)}}.$$ 
\end{definition}

This number coincides with the geometric angle between the straight lines connecting the midpoints of the two pairs of opposite edges in $J_1$ and $J_2$ when $C$ is thought of as a regular Euclidean polygon. The key observation is that this angle does depend on the cycle $C$ under consideration \cite[Proposition~3.20]{Ge_mediangle}, allowing us to define the \emph{angle} $\measuredangle(J_1,J_2)$ as the angle at an arbitrary even convex cycle crossed by our two hyperplanes.



Now, we continue with some properties of gated sets in mediangle graphs and some observations regarding the interaction between hyperplanes and gates on gated subgraphs. We start with the following characterization of gated subgraphs, which parallels a similar characterization of gated sets in weakly modular graphs (graphs satisfying the Triangle and Quadrangle Conditions):

\begin{lemma}[{\hspace{1sp}\cite[Proposition~6.5]{Ge_rotation}}]\label{lem:WhenGated}
A subgraph $H$ of a mediangle graph $G$ is gated if and only if $H$ is connected and satisfies the following two conditions:
\begin{itemize}
	\item every triangle with an edge in $H$ is entirely contained in $H$;
	\item every convex cycle with two consecutive edges in $H$ is entirely contained in $H$.
\end{itemize}
\end{lemma}


\begin{lemma}[{\hspace{1sp}\cite[Corollary~3.18]{Ge_rotation}}]\label{lem:HypSepProjs}
Let $G$ be a mediangle graph, $H$ a gated subgraph of $G$ and  $x,y$ two vertices of $G$. Then the 
hyperplanes separating $\gate_H(x)$ and $\gate_H(y)$ coincide with the hyperplanes that cross $H$ and separate $x$ and $y$. In particular, a hyperplane separating a vertex $x$ from its gate $\gate_H(x)$ is disjoint from $H$. 
\end{lemma}


Given two gated sets $A,B$, two vertices $a\in A$ and $b\in B$ are called \emph{mutual gates} if $a=\gate_A(b)$ and $b=\gate_B(a)$. Denote by $\gate_A(B)$ and $\gate_B(A)$ the sets of mutual gates of $B$ in $A$ and of $A$ in $B$, respectively. It is well-known \cite{DrSch} that for any graph $G$, the sets $\gate_A(B)$ and $\gate_B(A)$ 
induce two isomorphic gated subgraphs of $G$ and this isomorphism is defined by the pairs of mutual gates. Since the cliques in mediangle graphs are gated, applying this result and the previous lemma, we obtain: 

\begin{lemma}\label{lem:BijCliques}
Let $G$ be a mediangle graph and $K_1,K_2$ two cliques of $G$. Then either $\gate_{K_1}(K_2)$ and $\gate_{K_2}(K_1)$ are single vertices, or $\gate_{K_1}(K_2)=K_1$ and $\gate_{K_2}(K_1)=K_2$ and the mutual gate map induces a bijection between $K_1$ and $K_2$. Moreover, the latter case happens precisely when $K_1$ and $K_2$ belong to the same hyperplane. \qed
\end{lemma}

We conclude this subsection with the following characterization of sectors of hyperplanes via the gate map:

\begin{lemma}[{\hspace{1sp}\cite[Claim~3.15]{Ge_mediangle}}]\label{lem:SectorFromProj}
Let $G$ be a mediangle graph, $J$ a hyperplane of $G$, and $K$ a clique in $J$. Then the sectors delimited by $J$ are  the sets of the form  $\{ v\in V: \gate_K(v)=z\}$ for $z \in K.$ 
\end{lemma}

\subsection{Right hyperplanes}\label{section:RightHyp}

Median graphs coincide with mediangle graphs in which the angles between transverse hyperplanes are all $\pi/2$. This motivates the following definition, which introduces a family of hyperplanes that behave ``like in median graphs''.

\begin{definition}
Let $G$ be a mediangle graph. A hyperplane $J$ is \emph{right} if $\measuredangle(J,H) = \pi/2$ for every hyperplane $H$ transverse to $J$. 
\end{definition}

In other words, a right hyperplane is a hyperplane that crosses only even convex cycles of length $4$. As shown by our next lemma, right hyperplanes include all hyperplanes containing cliques with $\geq 3$ vertices.

\begin{lemma}[{\hspace{1sp}\cite[Lemma~2.24]{Ge_rotation}}]\label{lem:ThickRight}
Let $G$ be a mediangle graph and $K$ a clique. If $K$ contains at least three vertices, then the hyperplane containing $K$ is right.
\end{lemma}

A major difference between mediangle graphs and median graphs is that, in mediangle graphs, sectors may not be gated. Following our intuition that right hyperplanes behave like hyperplanes in median graphs, as a main result of this subsection, we will show:

\begin{proposition}\label{prop:RightHypGated}
Let $G$ be a mediangle graph and $J$ a right hyperplane. Sectors and fibers delimited by $J$ are gated. Moreover, given a clique $K$ contained in $J$ and a fiber $B$ delimited by $J$, the map
$$\left\{ \begin{array}{ccc} N(J) & \to & K \times B \\ x & \mapsto & (\gate_K(x), \gate_B(x)) \end{array} \right.$$
defines a graph-isomorphism. 
\end{proposition}

Proposition~\ref{prop:RightHypGated} will play a central role in the proof of Theorem~\ref{thm:GeneralCase} and the main goal of this subsection is to prove this proposition. Notice that we already know that:

\begin{lemma}[{\hspace{1sp}\cite[Lemma~2.26]{Ge_rotation}}]\label{lem:GatedCarrier}
In a mediangle graph $G$, right hyperplanes $J$ have gated neighborhoods $N(J)$.
\end{lemma}

We will also need the following two local properties of mediangle graphs, inspired by the $3$-cube condition satisfied by median graphs:

\begin{description}
	\item[(House Condition)] If $o,a,b,c,o'$ are vertices of $G$ such that $o\sim a,b,c, o\nsim o'$, and $o,b,c$ induce a triangle and $o,a,o',b$ induce a square,  then $\conv(o,a,b,c,o')$ is  a prism $K_3 \times K_2$. See~\Cref{fig:house}. 
    \medskip \noindent

\begin{figure}[ht]
    \centering

\includegraphics[width=0.17\linewidth]{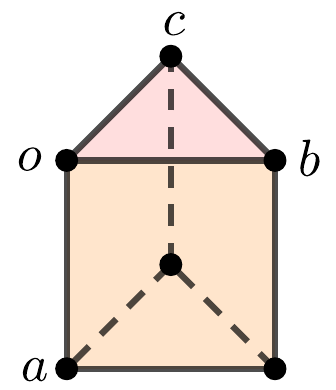}
\caption{The House Condition.}

\label{fig:house}
\end{figure}
    	\item[(Drumbell Condition)] If $o,a,b,c$ are vertices of $G$ such that $o\sim a,b,c$ and the edge $oa$ spans two squares, one with the edge $ob$ and one with the edge $oc$ and the edges $ob$ and $oc$ span a convex cycle $C$,  then $\{o,a,b,c\}$ is contained in a prism $C \times K_2$. See~\Cref{Drumbell}.
        
\begin{figure}[h!]
\begin{center}
\includegraphics[width=0.6\linewidth]{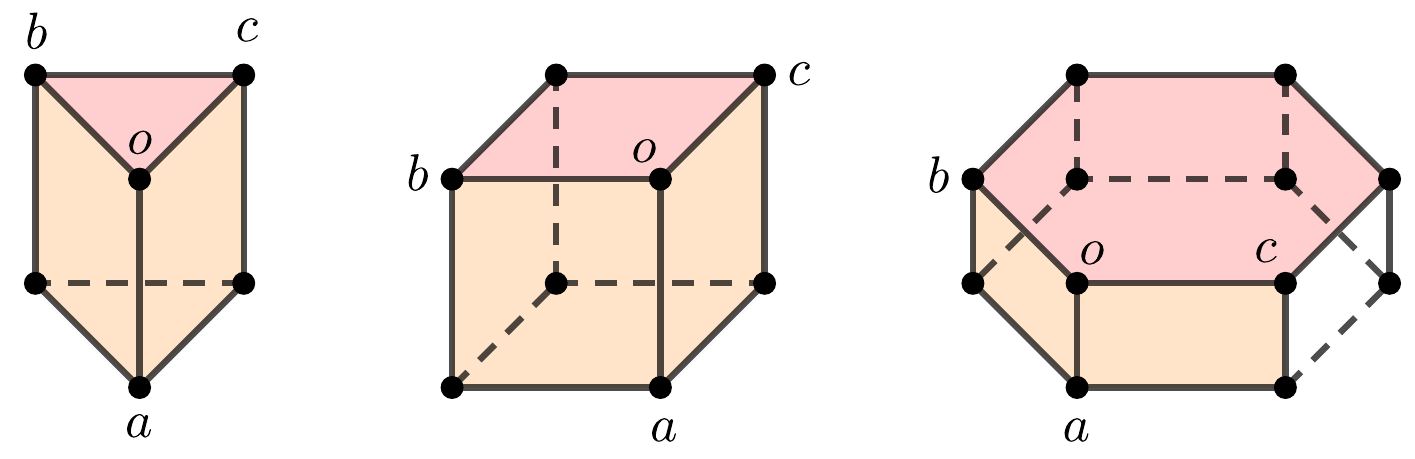}
\caption{Three examples illustrating the Drumbell Condition.}
\label{Drumbell}
\end{center}
\end{figure}
\end{description} 

\begin{lemma}\label{lem:HourseCondition}
Mediangle graphs satisfy the House Condition. 
\end{lemma}

\begin{proof}
If our house subgraph $\{o,o',a,b,c\}$ is not induced, then $c$ must be adjacent to $a$ or $o'$, say $c\sim a$ (the case $c\sim o'$ is similar). Let $K$ be a clique containing the triangle $oac$. By Theorem \ref{thm:BigHyp}(4), $K$ is gated. Since $b$ is adjacent to $o,c\in K$, the vertex $b$ must belong to $K$, yielding $b\sim a$, which is impossible because $oao'b$ is a square. 
Thus, our house subgraph is induced, and consequently isometrically embedded.  Then, we can apply the Triangle Condition to $c$ and $ao'$ in order to find the missing vertex $x$ that together with $o,o',a,b,c$ define the prism $K_3 \times K_2$. This prism is obviously convex, otherwise one of its squares will be not convex, contrary to Theorem \ref{convex-cycles->gated}.  
\end{proof}

\begin{lemma}\label{lem:Drumbell}
Mediangle graphs satisfy the Drumbell Condition. 
\end{lemma}

\begin{proof}
Let $o,a,b,c$ be four vertices of a mediangle graph $G$ as in the Drumbell Condition. We distinguish three cases, depending on the length of the convex cycle $C$ (see Figure~\ref{Drumbell}).  Let $b'$ (respectively, $c'$) denote the fourth vertex of the square $C'$ spanned by $oa$ and $ob$ (respectively, of the square $C''$ spanned by $oa$ and $oc$).

First, assume that $C$ has length 3. We assert that the vertices $b'$ and $c'$ are distinct and adjacent. If $b'=c'$, then since $o\nsim b',c'$, the vertices $o,b,c,b'=c'$ will induce a diamond, which is forbidden by the Diamond Condition. Hence $b'\ne c'$. Now, the vertices $o,a,b,b',c$ span a house. By the House Condition, their convex hull $\conv(o,a,b,b',c)$ is a prism $K_3\times K_2$. Since $c'$ belongs to  $\conv(o,a,b,b',c)$ (because $c'\in [a,c]$), necessarily $c'\sim b'$ as required. 

Next, assume that $C$ has length 4. If two of the edges $oa$, $ob$, $oc$ belong to the same hyperplane, 
then two of the vertices $a,b,c$ must be adjacent in $G$, contradicting that $C,C',C''$ are squares of $G$. 
Therefore, the edges  $oa$, $ob$, $oc$  belong to different hyperplanes and from Theorem~\ref{thm:BigHyp} we deduce that the union of $C,C',C''$ is isometrically embedded in $G$. Therefore, if $o'$ denotes the fourth vertex of $C$, then $d(a,o')=3$ 
and we can apply the Cycle Condition  to $o'$ and neighbors $b', c'\in [a,o']$ of $a$ to derive the missing vertex $x$ of our prism $C \times K_2$ containing $\{o,a,b,c\}$. From the Diamond Condition,  $x$ cannot be adjacent to any of the vertices $b,c,a$, whence this prism is an induced subgraph of $G$. 

Finally, assume that $C$ is a convex cycle of length $>4$. By Theorem \ref{convex-cycles->gated}, $C$ is gated. Since $C$ is isometrically embedded, the hyperplanes crossing $C$ must be pairwise distinct. On the other hand,  the gate of $a$ in $C$ is $o$, which implies that $d(a,o') = 1+ \mathrm{length}(C)/2$ where $o'$ is the vertex of $C$ opposite to $o$. Consequently, the hyperplane containing the edge $oa$ is not crossing $C$. On the other hand, the hyperplanes containing $ab'$ and $ac'$ are the same as the hyperplanes containing $ob$ and $oc$, and thus  are distinct. This proves that the union of $C$ and the two squares $C',C''$ is isometrically embedded. Therefore, we can apply the Cycle Condition to $o'$ and neighbors $b',c'\in [a,o']$ of $a$ and find a convex cycle $Q$ spanned by the edges $ab'$ and $ac'$. By \cite[Proposition~3.21]{Ge_rotation} the restriction of the gate map $Q \to C$ is a graph-isomorphism, which implies that the subgraph induced by $C \cup Q$ is isomorphic to $C \times K_2$. 
\end{proof}

Finally, we record a last preliminary lemma before turning to the proof of Proposition~\ref{prop:RightHypGated}. 

\begin{lemma}\label{lem:SquareCarrier}
Let $G$ be a mediangle graph, $J$ a right hyperplane, and $xy$ an edge contained in a fiber $B$ delimited by $J$. For every edge $xx'$ in $J$, there exists an edge $yy'$ in $J$ such that $\{x,x',y,y'\}$ induces a square.
\end{lemma}

\begin{proof}
Since $y\in B$, there exists a clique $K$ of $J$ containing $y$. Since $G$ is clique-gated, let $y'$ denote the gate of $x'$ in $K$. Since $y$ is the gate of $x$ in $K$, Lemma~\ref{lem:BijCliques} implies that $y' \neq y$. Since $d(x',y) = 2$, this implies  $d(x',y')=d(y,y')=1$. Since $y$ and $y'$ are the gates of $x$ and $x'$ in $K$, we get $d(x,y') =d(x,y)+d(y,y') =2$ and $d(x',y)=d(x',y')+d(y',y)=2$. Thus, $\{x,x',y,y'\}$ defines a square, as desired. 
\end{proof}

\begin{proof}[Proof of Proposition~\ref{prop:RightHypGated}.]
Let $B$ be a fiber delimited by $J$. First, using the two conditions of Lemma~\ref{lem:WhenGated},  we show that  $B$ is gated.  Let $T$ be a triangle of $G$ with an edge $xy$ in $B$. According to Lemma~\ref{lem:SquareCarrier}, $x$ and $y$ respectively belong to edges $xx'$ and $yy'$ of $J$ that are opposite in some square $S$. By applying the House Condition (Lemma~\ref{lem:HourseCondition}) to $T \cup S$, it follows that the third vertex of $T$ also belongs to $B$, as desired. 

Let $C$ be an even convex cycle of $G$ with two consecutive edges $xy$ and $yz$ in $B$. According to Lemma~\ref{lem:SquareCarrier}, the vertices $x,y,z$  belong to edges $xx'$, $yy'$, $zz'$, respectively, such that $\{x,x',y,y'\}$ and $\{y,y',z,z'\}$ induce squares $S$ and $S'$ respectively. By applying the Drumbell Condition (Lemma~\ref{lem:Drumbell}) to $C \cup S \cup S'$, it follows that all vertices of $C$ lie in $B$, as desired. Thus, Lemma~\ref{lem:WhenGated} applies and shows that $B$ is gated.

Now, let $S$ denote the sector delimited by $B$. From the gatedness of $B$ we can deduce that $S$ is gated. Indeed, let $x$ be an arbitrary vertex of $G$. We claim that $x$ has a gate in $S$. If $x$ belongs to $S$, this is clear, so we assume that $x$ does not belong to $S$. Let $p$ denote the gate of $x$ in $B$. For every vertex $y$ of  $S$, a geodesic connecting $x$ to $y$ has to intersect $B$, hence
$$d(x,y)= d(x,q)+d(q,y) \text{ for some } q \in B.$$
But we know that $d(x,q)=d(x,p)+d(p,q)$, hence
$$d(x,y)=d(x,p)+d(p,q)+d(q,y) \geq d(x,p)+d(p,y).$$
The reverse inequality follows from triangle inequality, so  $d(x,y)=d(x,p)+d(p,y)$. Thus, $p$ is also the gate of $x$ in $S$. 

Consequently,  sectors and fibers of  right hyperplane $J$ are gated. It remains to verify that the neighborhood $N(J)$ decomposes as a product. Pick a clique $K$ in $J$ and a fiber $B$ delimited by $J$. Our goal is to verify that the map
$$\Lambda : \left\{ \begin{array}{ccc} N(J) & \to & K \times B \\ x & \mapsto & (\gate_K(x), \gate_B(x)) \end{array} \right.$$
defines a graph-isomorphism. We start by recording a couple of observations:

\begin{claim}\label{claim:CliqueInterFibre}
For all cliques $K$ in $J$ and all fibers $B$ of $J$, the intersection $K \cap B$ is non-empty. 
\end{claim}

By definition, the fibers of $J$ are the connected components of $N(J) \backslash J$. They coincide with the intersection between $N(J)$ and the connected components of $G\backslash J$, namely the sectors delimited by $J$.  This is so because the intersection between $N(J)$ and any sector is connected as an intersection of two gated subgraphs. Thus, in order to prove our claim, it suffices  to verify that every clique $K$ in $J$ intersects every sector delimited by $J$. But this follows from Lemma~\ref{lem:SectorFromProj}, concluding the proof of Claim~\ref{claim:CliqueInterFibre}.

\begin{claim}\label{claim:HypCarrier}
The set of the hyperplanes crossing $N(J)$ is $\{J\} \sqcup \{ \text{hyperplanes crossing } B\}$. 
\end{claim} 

Let $H$ be a hyperplane crossing $N(J)$. Fix an edge $xy$ of $H$ that belongs to $N(J)$. If $xy$ belongs to $J$, then there is nothing to prove. So, from now on, we assume that $xy$ does not belong to $J$. The vertex $x$ belongs to some clique $K$ of $J$, which intersects $B$ according to Claim~\ref{claim:CliqueInterFibre}. In other words, $x$ has a neighbor $x'$ that belongs to $B$. According to Lemma~\ref{lem:SquareCarrier}, $y$ has a neighbor $y'$ such that $\{x,x',y,y'\}$ defines a square. Then, $x'y'$ is an edge of $H$ that is contained in $B$, proving that $H$ crosses $B$. Conversely, it is clear that $J$ and every hyperplane crossing $B$ always crosses $N(J)$. Thus, Claim~\ref{claim:HypCarrier} is proved. 

Let $x,y$ be two vertices in $N(J)$. By Theorem~\ref{thm:BigHyp}, the distance $d(x,y)$ between $x$ and $y$ coincides with the number of hyperplanes separating them. Therefore, it follows from Claim~\ref{claim:HypCarrier} and Lemma~\ref{lem:HypSepProjs} that
$$\begin{array}{lcl} d(x,y) & = & \# \{ \text{hyperplanes separating $x$ and $y$} \} \\ \\ & = & \# \left\{ \begin{array}{c} \text{hyperplanes of $B$} \\ \text{separating $x$ and $y$} \end{array} \right\} + \# \left\{ \begin{array}{c} \text{hyperplanes of cliques $K$} \\ \text{separating $x$ and $y$} \end{array} \right\} \\ \\ & = & d \left( \gate_B(x), \gate_B(y) \right) + d \left( \gate_K(x), \gate_K(y) \right) \end{array}.$$
In other words, the map $\Lambda$ induces an isometric embedding $N(J) \to K \times B$. It remains to verify that it is also surjective on vertices. So fix a vertex $(c,p) \in K \times B$. The vertex $p$ of $B$ must belong to some clique $K'$ of $J$. We know from Lemma~\ref{lem:BijCliques} that the projection to $K$ induces a bijection $K' \to K$, so $K'$ contains some vertex $z$ whose gate in $K$ is $c$. Since $d(z,p) \leq 1$, it is clear that the gate of $z$ in $B$ is $p$. Thus, $\Lambda(z)=(c,p)$, as desired. 
\end{proof}

As a consequence of Proposition~\ref{prop:RightHypGated}, we obtain:

\begin{corollary}\label{cor:FinitelyManyRight}
Let $G$ be a mediangle graph and $R$ be a finite set of vertices. The gated hull $\lgate R\rgate$ of $R$ is crossed by only finitely many right hyperplanes.
\end{corollary}

\begin{proof}
Let $J$ be a right hyperplane. If $J$ does not separate two vertices in $R$, then $R$ is contained in a  single sector $S$ delimited by $J$. By Proposition~\ref{prop:RightHypGated} $S$ is gated, whence $\lgate R\rgate\subseteq S$. Thus, every right hyperplane crossing $\lgate R\rgate$ must separate at least two vertices of $R$. Since $R$ is finite, only finitely many right hyperplanes crossing $\lgate R\rgate$ may exist. 
\end{proof}

We emphasize that, in a mediangle graph, the gated hull of finitely many vertices may not be finite. For instance, the regular hexagonal tiling of the plane defines a mediangle graph in which the only proper gated subgraphs are the vertices, the edges, and the hexagons, which implies that generic subgraphs have infinite gated hulls.

\subsection{Proof of Theorem \ref{thm:MediangleGatedAntipod}}\label{section:GatedAnti} 
We start by defining gated-antipodal subgraphs of mediangle graphs.  

\begin{definition}\label{def:gatan}
A subgraph $A$ of a mediangle graph $G$ is called \emph{gated-antipodal} if, for every vertex $a$ of $A$, there exists an $a'\in A$ such that $A$ coincides with the gated hull $\lgate a,a'\rgate$ of $\{ a,a'\}$ and all the hyperplanes of $A$ separate $a$ and $a'$. 
\end{definition}

We continue with some properties of antipodal-gated subgraphs. 

\begin{lemma}\label{lem:InterAntipod}
Let $G$ be a mediangle graph, $H$ be a gated subgraph of $G$, and $A$ be a gated-antipodal subgraph of $G$. Then the intersection $A \cap H$ is gated-antipodal.
\end{lemma}

\begin{proof}
Pick any vertex $x$ of $A \cap H$. Since $A$ is gated-antipodal, there exists a vertex $x'$ such that $A$ coincides with the gated hull $\lgate x,x'\rgate$ of $x,x'$ and such that all the hyperplanes crossing $A$ separate $x$ and $x'$. Let $x''$ denote the gate of $x'$ in $A \cap H$.  First, we show that the hyperplanes crossing $A \cap H$ coincide with the hyperplanes separating $x$ and $x''$. Clearly, every hyperplane separating $x$ and $x''$ crosses $A \cap H$. Conversely, let $J$ be a hyperplane crossing $A \cap H$. On the one hand, since $J$ crosses $A$, from the definition of $x'$ it follows that $J$  has to separate $x$ and $x'$. Since $J$ crosses $H$, it follows from Lemma~\ref{lem:HypSepProjs} that $J$ does not separate $x'$ and $x''$. Consequently, $J$ separates $x$ and $x''$, as required.

Since $A$ and $H$ are gated and $x,x''\in A\cap H$, clearly  the gated hull $\lgate x,x''\rgate$  of $x,x''$ is contained in $A \cap H$. If the inclusion is strict, then we can find a vertex $p$ that belongs to $A \cap H$ but not to  $\lgate x,x''\rgate$. Taking a hyperplane $J$ separating $p$ from its gate in  $\lgate x,x''\rgate$, we know from Lemma~\ref{lem:HypSepProjs} that $J$ does not cross  $\lgate x,x''\rgate$. Since $J$ crosses $A\cap H$, from our previous observation it follows that $J$ has to separate $x$ and $x''$, a contradiction. 
\end{proof}

Our next result shows that  in the case of bipartite mediangle graphs gated-antipodal subgraphs are exactly the antipodal subgraphs: 

\begin{proposition}\label{prop:GatedAntVsAnt}
Let $G$ be a bipartite mediangle graph. A subgraph $A$ of $G$ is gated-antipodal if and only if it is antipodal. 
\end{proposition}

The proof of the proposition is based on the following lemma: 

\begin{lemma}\label{lem:IntGated}
Let $G$ be a bipartite mediangle graph and $p,q$ two vertices. If the interval $[p,q]$ is not gated, then in the gated hull $\lgate p,q\rgate$ of $p,q$, there is some hyperplane not separating $p$ and $q$. 
\end{lemma}

\begin{proof}
Notice that, since bipartite mediangle graphs are partial cubes, we already know that $[p,q]$ is convex. According to \cite[Proposition~6.5]{Ge_mediangle}, the fact that $[p,q]$ is not gated implies that there exists some gated cycle $C$ with two consecutive edges in $[p,q]$, say $ab$ and $ac$, but not entirely contained in $[p,q]$.

First, assume that $d(p,b)=d(p,c)=d(p,a)-1$. Then, we know from the Cycle Condition that the edges $ab$ and $ac$ span a convex cycle $Q$ such that the vertex $a'$ opposite to $a$ in $Q$ belongs to $[p,b] \cap [p,c]$. Clearly, $Q$ is contained in $[p,a] \subset [p,q]$. But, since both $Q$ and $C$  contain $ab$ and $ac$ and $C$ is gated, we must have $C=Q$, hence $C \subset [p,q]$, contrary to the choice of $C$. So this case cannot happen. Next, assume that $d(p,b)=d(p,c)=d(p,a)+1$. Since $C\subseteq [p,q]$ and $C$ is gated, we deduce that $d(q,b)=d(q,c)=d(q,a)-1$. So the same argument as before with $q$ instead of $p$ shows that this case cannot happen either.

Up to switching $b$ and $c$, it remains only one case when $|d(p,b)-d(p,c)|=2$ and $|d(q,b)-d(q,c)|=2$. Suppose without loss of generality that $d(p,b)=d(p,a)-1$ and $d(p,c)=d(p,a)+1$. Let $p'$ (respectively, $q'$) denote the gate of $p$ (respectively,  $q$) in $C$. If $p'$ and $q'$ are not opposite in $C$, then there exists a hyperplane $J$ crossing $C$ but not separating $p'$ and $q'$. Since we know, as a consequence of Lemma~\ref{lem:HypSepProjs}, that $J$ cannot separate $p$ from $p'$ nor $q$ from $q'$, it follows that $J$ does not separate $p$ and $q$. Thus, $J$ is the hyperplane we are looking for. So, from now on, we assume that $p'$ and $q'$ are opposite vertices in $C$. First suppose that $d(p,q)=d(p,p')+d(p',q')+d(q',q)$. Since $p'$ and $q'$ are opposite vertices of $C$, this implies that $C$ is contained in $[p,q]$, a contradiction. 

Finally,  suppose that $d(p,q)<d(p,p')+d(p',q')+d(q',q)$. We assert that in this case there must be a hyperplane separating $\{p',q'\}$ from $\{p,q\}$. Since $G$ is a partial cube, $d(p,q)$ is equal to the number of hyperplanes separating $p$ and $q$. 
First, pick any hyperplane $L$ separating $p$ and $p'$. Since $C$ gated and $p'$ is the gate of $p$ in $C$, the hyperplane $L$ also separates $p$ from $C$ and thus from $\{ p',q'\}$. If $L$ does not separate $p$ and $q$, then $L$ is the desired hyperplane and we are done. Therefore, we can suppose that any hyperplane $J$ separating $p$ from $p'$ (and $p$ from $\{ p',q'\}$) also separates $p$ from $q$. Analogously, we can suppose that any hyperplane $J$ separating $q$ from $q'$ (and $q$ from $\{ p',q'\}$) also separates $q$ from $p$. Finally, from all this it also follows that any hyperplane separating $p'$ and $q'$ also separates $p$ and $q$. This implies that $d(p,q)=d(p,p')+d(p',q')+d(q',q)$, a contradiction. Consequently, there must be a hyperplane $H$ separating $\{p',q'\}$ from $\{p,q\}$. Notice that, again as a consequence of Lemma~\ref{lem:HypSepProjs}, $H$ 
does not cross $C$. Thus, $H$ separates $C$ from $\{p,q\}$, contradicting the fact that $C$ intersects $[p,q]$. This concludes the proof. 
\end{proof}

\begin{proof}[Proof of Proposition~\ref{prop:GatedAntVsAnt}.]
Let $A$ be an antipodal subgraph. We know that a given vertex $a$ of $A$ admits an opposite vertex, say $\bar{a}$, i.e., $A$ coincides with the interval $[a,\bar{a}]$. Because sectors are convex, every hyperplane of $A$ has to separate $a$ and $\bar{a}$. Then, it is clear that the gated hull $\lgate a,\bar{a}\rgate$ of $a,\bar{a}$ contains the convex hull $\conv(a,\bar{a})$ of $a,\bar{a}$, and a fortiori the interval $[a,\bar{a}]$, namely $A$. Since antipodal  subgraphs in tope graphs of COMs are gated \cite{KnMa}, it follows from Theorem~\ref{t:mediangle->COM} that $A$ is gated. Therefore, we know that, conversely, $A$ must contain the gated hull $\lgate a,\bar{a}\rgate$. We conclude that $A$ coincides with $\lgate a,\bar{a}\rgate$, thus  $A$ is gated-antipodal.

Conversely, assume that $A$ is gated-antipodal. Then  for any  $a\in A$ there exists a vertex $\bar{a}\in A$ such that $A$ coincides with the gated hull $\lgate a,\bar{a} \rgate$ and all the hyperplanes of $A$ separates $a$ and $\bar{a}$. Since  $\lgate a,\bar{a}\rgate$ contains the interval $[a,\bar{a}]$, necessarily $[a,\bar{a}] \subseteq A$. But from Lemma~\ref{lem:IntGated} we deduce that $[a,\bar{a}]$ is gated, hence $A \subset [a,\bar{a}]$. Thus, we have proved that $A=[a,\bar{a}]$, and finally that $A$ is antipodal, as desired.
\end{proof}

\begin{lemma}\label{l:extended}
Let $J$ be a right hyperplane of a mediangle graph $G$ and $S$ be a sector delimited by $J$. Then the extended sector  $S^+=S \cup N(J)$ is gated.
\end{lemma}

\begin{proof}
Let $x$ be a vertex not in $S^+$. We claim that $\gate_{N(J)}(x)$ is also the gate of $x$ in $S^+$. Let $y$ be an arbitrary vertex of $S^+$. If $y$ belongs to $N(J)$, then it is clear that $d(x,y)=d(x,\gate_{N(J)}(x))+d(\gate_{N(J)}(x),y)$. Otherwise, if $y$ belongs to $S$ but not to $N(J)$, then $J$ separates $x$ and $y$, so a geodesic connecting $x$ to $y$ must intersect $N(J)$, say at a vertex $z$. Then,
$d(x,y)=d(x,z)+d(z,y)=d(x,\gate_{N(J)}(x))+d(\gate_{N(J)}(x),z)+d(z,y).$
In order to conclude, it remains to verify that $d(\gate_{N(J}(x),z)+d(z,y)=d(\gate_{N(J)}(x),y)$. But, if the equality does not hold, it follows from Theorem~\ref{thm:BigHyp} that some hyperplane $J'$ separates $z$ from $\gate_{N(J)}(x)$ and from $y$.  Since $z$ lies on a geodesic between $x$ and $y$, $J'$ cannot also separate $x$ and $z$, so it has to separate $x$ and $\gate_{N(J)}(x)$. But we know from Lemma~\ref{lem:HypSepProjs} that every hyperplane separating $x$ from $\gate_{N(J)}(x)$ does not cross $N(J)$. On the other hand, $J'$ does cross $N(J)$ since it separates $\gate_{N(J)}(x)$ and $z$, yielding a contradiction.
\end{proof} 

We continue with the main result of this subsection, which shows that gated-antipodal subgraphs in mediangle graphs indeed define cells. 

\begin{lemma}\label{lem:AntipodSub}
In a mediangle graph $G$, a gated-antipodal subgraph $A$ is a Cartesian product of  finitely many cliques and an antipodal bipartite mediangle graph.
\end{lemma}

\begin{proof}
Let $G$ be a mediangle graph and $A$ be a  gated-antipodal subgraph of $G$. Since $A$ coincides with the gated hull of at least two of its vertices, Corollary~\ref{cor:FinitelyManyRight} implies that $A$ has only finitely many right hyperplanes. We prove our lemma  by induction on the number of right hyperplanes of $A$. If $A$ does not have any right hyperplane, then from Lemma~\ref{lem:ThickRight} we deduce that $A$ does not contain cliques of size at least 3. Thus $A$ is bipartite, 
and by Proposition~\ref{prop:GatedAntVsAnt} we deduce that $A$ is an antipodal  subgraph of $G$ and we are done. 
Now, we assume that $A$ contains at least one right hyperplane, say $J$. 

If $A=N(J)$, then it follows from Proposition~\ref{prop:RightHypGated} that $A= K \times B$ for a clique $K$ of $J$ and a fiber $B$ of $J$. Notice that $B$ has less right hyperplanes than $A$, since $J$ crosses $A$ but not $B$. Moreover, since $B$ coincides with the intersection of $A$ and a sector delimited by $J$, which is gated according to Proposition~\ref{prop:RightHypGated}, from Lemma~\ref{lem:InterAntipod} we deduce that $B$ is gated-antipodal. Consequently, it follows by induction that $B$, and a fortiori $A$, is a Cartesian product of finitely many cliques and an antipodal bipartite mediangle graph.

Otherwise, if $N(J)$ is a proper subset of $A$, then we can find a vertex $a \in A\setminus N(J)$. Let $p$ denote its gate in $N(J)$. Since $A$ is gated-antipodal, there must exist some vertex $p' \in A$ such that $A$ coincides with the gated hull $\lgate p,p'\rgate$ and all hyperplanes crossing $A$ separate $p$ and $p'$. Clearly, $p'$ cannot belong to the same sector $S$ delimited by $J$ as $p$, since otherwise $\lgate p,p'\rgate$, and therefore $A$, would be contained in $S$ (as $S$ is gated according to Proposition~\ref{prop:RightHypGated}). Thus, $p'$ belongs to a sector $S'$ delimited by $J$ that does not contain $p$ nor $a$. By Lemma \ref{l:extended} it follows that $N(J) \cup S'$ is a gated subgraph of $A$ containing $p, p'$ and not containing $a$. This contradicts the assumption that $\lgate p,p'\rgate=A$ and finishes the proof of the lemma.
\end{proof}


A \emph{tesselation of a ball} is a regular cell complex with a single maximal cell. 

\begin{corollary}\label{cor:Cell}
A gated-antipodal mediangle graph is the $1$-skeleton of a tesselation of a ball.
\end{corollary}

\begin{proof}
This follows immediately from the description provided from Lemma~\ref{lem:AntipodSub} combined with Corollary~\ref{cor:BipartiteCase}. 
\end{proof}

Finally, Theorem \ref{t:cells-in-mediangle}  follows by combining Lemma~\ref{lem:AntipodSub} and Theorem \ref{t:cells-in-mediangle}.

\subsection{Proof of Theorem~\ref{thm:GeneralCase}}\label{section:ProofGeneral}

We are finally ready to prove the main result of this paper, namely Theorem~\ref{thm:GeneralCase}.
%
By Corollary~\ref{cor:Cell} and Lemma~\ref{lem:InterAntipod}, $G$ is the $1$-skeleton of a cell complex $X(G)$ obtained by filling every gated-antipodal subgraph of $G$ with a tesselation of a ball. It remains to verify that $X(G)$ is contractible. In fact, as a consequence of Whitehead Theorem, it suffices to verify that $X(G)$ is aspherical. Since (1) a continuous image of a sphere in $X(G)$ intersects only finitely many cells, (2) the gated hull of finitely many vertices is crossed by only finitely many right hyperplanes (Corollary~\ref{cor:FinitelyManyRight}), and (3) a gated-antipodal subgraph in a gated subgraph is gated-antipodal in $G$, we can assume without loss of generality that $G$ has only finitely many right hyperplanes. Any gated subgraph $H$ of a mediangle graph $G$ is mediangle, therefore the cell complex of $H$ is well-defined and will be denoted by $X(H)$. 

\begin{claim}
For every mediangle graph $G$, if $G$ has only finitely many right hyperplanes then $X(G)$ is contractible. 
\end{claim}

We argue by induction on the number of right hyperplanes of $G$. If $G$ has no right hyperplane, then it follows from Lemma~\ref{lem:ThickRight} that $G$ is bipartite, so the desired conclusion follows from Corollary~\ref{cor:BipartiteCase}. From now on, assume that $G$ contains at least one right hyperplane. Pick any such a hyperplane, say $J$.  If $\{S_i: i \in I\}$ denotes the set of sectors delimited by $J$ and if we denote by $B_i$ the fibers delimited by $J$ contained in $S_i$ for every $i \in I$, then $G$ can be described as the disjoint union $N(J) \sqcup \bigsqcup_{i \in I} S_i$ and $G$ is obtained by gluing each $S_i$ to $N(J)$ along $B_i$. Notice that, since we know from Proposition~\ref{prop:RightHypGated} that the $S_i$ and $B_i$ are gated subgraphs of $G$, the gated-antipodal subgraphs of  $S_i$ and $B_i$ are gated-antipodal subgraphs of $G$. Furthermore, by  Lemma \ref{lem:InterAntipod} every gated-antipodal subgraph $A$ of $G$ intersects each $S_i$ and $B_i$ along a gated-antipodal subgraph. Therefore, the cell complex $X(G)$ of $G$ can also be described as obtained from the disjoint union $X(N(J)) \sqcup \bigsqcup_{i \in I} X(S_i)$ by gluing along $X(B_i)$ the cell complex $X(S_i)$ of each sector $S_i$ to the cell complex $X(N(J))$ of $N(J)$.

Notice that the gated subgraphs $S_i$ and $B_i$ have fewer right hyperplanes than $G$, since they are not crossed by $J$. Since $S_i$ and $B_i$ are mediangle, it follows by induction assumption that their cell complexes $X(S_i)$ and $X(B_i)$ are contractible. Since a contractible subcomplex in a contractible complex is automatically a deformation retract (this is a basic observation of algebraic topology, see for instance \cite[Theorem 4.23, Chapter 4]{Hat} or \cite[Fact~3.12]{HomotopyRAAG} for details), we know that each $X(S_i)$ deformation retracts onto $X(B_i)$. Since the cell complexes $X(S_i)$ are pairwise disjoint, all these deformation retractions can be applied simultaneously, showing that $X(G)$ deformation retracts onto $X(N(J))$. But it follows from Proposition~\ref{prop:RightHypGated} that all  $X(B_i),i \in I$ are pairwise isomorphic and that $X(N(J))$ decomposes as a product of a (possibly infinite) simplex and an arbitrary $X(B_i)$. Since we already know that $X(B_i)$ is contractible, we conclude that $X(N(J))$ is contractible as well. Thus, we have proved that $X(G)$ is contractible, as desired. This finishes the proof of Theorem~\ref{thm:GeneralCase}.

\section{Final remarks}\label{sec:fin} 

We proved that mediangle graphs are $1$-sekeleta of contractible cell complexes, whose cells are products of simplices and simplicial oriented matroids. In the bipartite case, the mediangle graphs are tope graphs of finitary Complexes of Oriented Matroids (COMs). We wonder if this rich combinatorial description can be extended beyond the bipartite case.

\begin{question}{\sf (Beyond bipartite)}
    Is there a suitable generalization of COMs whose 1-skeleta capture  all mediangle graphs?
\end{question} 


Towards further structural results on cell complexes of mediangle graphs, one can ask for other characterizations of cells. A relevant question in this direction is:

\begin{question} {\sf (Structure of cells)}
Is a mediangle graph gated-antipodal if and only if its hyperplanes are pairwise transverse?
\end{question}

In another direction, recall that for any two vertices $z$ and $u$ of a median graph $G$ and any set $S\subset [z,u]$ of $k$ of neighbors of $z$, there exists a $k$-dimensional cube $Q$ containing all vertices of $S$ and contained in $[z,u]$. This is called the \emph{downward cube property} and was first  established by Mulder \cite{Mu}. In \cite{CKM20}, this property was generalized to hypercellular graphs, a class of mediangle graphs~\cite{Ge_mediangle} that contains median graphs. We wonder if a similar property holds for general mediangle graphs: 

\begin{question}{\sf (The downward cell property)}
\label{t:downward-cell-property} 
Let $G$ be a mediangle graph. Is it true that for any  $z,u\in V$ and any set $S\subseteq N(z)\cap [z,u]$ there exists a cell $C$ such that $C\subseteq [u,z]$ and $N(z)\cap C=S$? 
\end{question}

For automorphism groups of median graphs and more generally hypercellular graphs and non-expansive mappings, in \cite{CKM20} we established a fixed cell property. This leads to a variant of~\cite[Question 7.2]{Ge_mediangle}:

\begin{question}  {\sf (The fixed cell property)}  If a group acts on a  mediangle graph with bounded (diameter) orbits, does it stabilize a cell? Does any non-expansive map from a  mediangle graph $G$ to itself fix a cell of $G$?
\end{question}

Our last set of questions concern the CAT(0) structure of mediangle graphs.  In the general form this is the content of the second part of Question 7.1 of \cite{Ge_mediangle}. 
The core of the question seems to lie in the bipartite case to which we restrict from now on. The main issue is that there exist non-realizable simplicial OMs, see~\cite{Gru03}. Hence not all mediangle graphs are 1-sekeleta of a zonotopal complex. Thus, it would be necessary to decide if the face lattice of a simplicial OM is realized by a Euclidean polytope. While this polytopality problem is hard for general face lattices \cite{Ric97}, it might be easier in this case. Note that since the tope graph of a COM determines the COM, it suffices to ask this question for tope graphs and not for face lattices:

\begin{question} \label{simplicialOMCAT(0)}  {\sf (Polytopality)} Let $G$ be the tope graph of a  simplicial OM. 
Does there exists a polytope $P$ in the Euclidean  space whose 1-skeleton is $G$? 
\end{question}

Even if Question \ref{simplicialOMCAT(0)} had a positive answer, it remains to glue the resulting polytopal cells by isometries to get the cell complex of a mediangle graph and then study the CAT(0) question here.
Already in the case when the resulting cell complex is well-defined and all cells are zonotopes, several questions remain open. In \cite{BaChKn}, we showed that the CAT(0) Coxeter (zonotopal) complexes introduced by Haglund and Paulin \cite{HaPau}  (see also \cite{Davis}) arise from COMs. Since the cells of Coxeter complexes are simplicial OMs, it is likley that their tope graphs are mediangle.  More generally, one may ask if the 1-sekeleta of CAT(0) complexes in which all cells are realizable simplicial OMs are mediangle: 

\begin{question} \label{CAT(0)zonotopal}  {\sf (CAT(0) zonotopal complexes and mediangle graphs)} Are the tope graphs of CAT(0) Coxeter complexes mediangle? 
\end{question}



\section*{Acknowledgements}
The work presented here was partially supported by the ANR project MIMETIQUE (ANR-25-CE48-4089-01). 
Kolja Knauer was supported by 
the grant PID2022-137283NB-C22 funded by MICIU/AEI/10.13039/501100011033 and by ERDF/EU and through the Severo Ochoa and María de Maeztu Program for Centers and Units of Excellence in R\&D (CEX2020-001084-M).

\bibliographystyle{abbrv}
\bibliography{refs-semigeometries}

@article{HomotopyRAAG,
  title={Homotopy types of complexes of hyperplanes in quasi-median graphs and applications to right-angled {A}rtin groups},
  author={Abbott, Carolyn and Genevois, Anthony and Mart\'{i}nez-Pedroza, Eduardo},
  journal={arxiv:2503.08411},
  year={2025}
}

@article{Ge_rotation,
 author = {Genevois, Anthony},
 title = {Rotation groups virtually embed into right-angled rotation groups},
 journal={Canadian Journal of Math. (to appear)},
 year = {2026},
 howpublished = {{arXiv}:2404.15652 [math.{GR}] (2024)},
 url = {https://arxiv.org/abs/2404.15652},
 arXiv = {arXiv:2404.15652}
}

@book{and25,
 author = {Anderson, Laura},
 title = {Oriented {M}atroids},
 fseries = {Cambridge Studies in Advanced Mathematics},
 series = {Camb. Stud. Adv. Math.},
 volume = {216},
 isbn = {978-1-00-949411-3; 978-1-00-949407-6},
 year = {2025},
 publisher = {Cambridge: Cambridge University Press},
 language = {English},
 doi = {10.1017/9781009494076},
 zbMATH = {7990079}
}

@article{ziegler2024oriented,
  title     = {Oriented Matroids Today},
  author    = {G{\"u}nter M. Ziegler and Laura Anderson and Kolja Knauer},
  journal   = {The Electronic Journal of Combinatorics},
  volume    = {Dynamic Surveys},
  number    = {DS4},
  year      = {2024},
  doi       = {10.37236/25},
  url       = {https://doi.org/10.37236/25}
}

@book{Ric97,
 author = {Richter-Gebert, J{\"u}rgen},
 title = {Realization spaces of polytopes},
 fseries = {Lecture Notes in Mathematics},
 series = {Lect. Notes Math.},
 issn = {0075-8434},
 volume = {1643},
 isbn = {3-540-62084-2},
 year = {1997},
 publisher = {Berlin: Springer},
 language = {English},
 doi = {10.1007/BFb0093761},
 zbMATH = {967844},
 Zbl = {0866.52009}
}

@article{Ge_mediangle,
 author = {Genevois, Anthony},
 title = {Rotation groups, mediangle graphs, and periagroups: a unified point of view on {Coxeter} groups and graph products of groups},
 fjournal = {Groups, Geometry, and Dynamics (to appear)},
 year={2026},
 journal = {Groups Geom. Dyn. (to appear)},
 howpublished = {{arXiv}:2212.06421 [math.{GR}] (2022)},
 url = {https://arxiv.org/abs/2212.06421},
 arXiv = {arXiv:2212.06421}
}

@article{De,
 author = {Deligne, Pierre},
 title = {Les immeubles des groupes de tresses g{\'e}n{\'e}ralises},
 fjournal = {Inventiones Mathematicae},
 journal = {Invent. Math.},
 issn = {0020-9910},
 volume = {17},
 pages = {273--302},
 year = {1972},
 language = {French},
 doi = {10.1007/BF01406236},
 url = {https://eudml.org/doc/142173},
 zbMATH = {3377490},
 Zbl = {0238.20034}
}

@article{BjEdZi,
 author = {Bj{\"o}rner, Anders and Edelman, Paul H. and Ziegler, G{\"u}nter M.},
 title = {Hyperplane arrangements with a lattice of regions},
 fjournal = {Discrete \& Computational Geometry},
 journal = {Discrete Comput. Geom.},
 issn = {0179-5376},
 volume = {5},
 number = {3},
 pages = {263--288},
 year = {1990},
 language = {English},
 doi = {10.1007/BF02187790},
 url = {https://eudml.org/doc/131117},
 zbMATH = {4144836},
 Zbl = {0698.51010}
}

@article{BaChKn,
  Author =	 {Bandelt, Hans-J{\"{u}}rgen and Chepoi, Victor and
                  Knauer, Kolja},
  Doi =		 {10.1016/j.jcta.2018.01.002},
  Fjournal =	 {Journal of Combinatorial Theory, Series {A}},
  Journal =	 {J. Combin. Theory Ser. {A}},
  Pages =	 {195--237},
  Sauthor =	 {Bandelt, H.-J. and Chepoi, V. and Knauer, K.},
  Title =	 {{COM}s: Complexes of Oriented Matroids},
  Volume =	 156,
  Year =	 2018
}

@article{BaCh_cellular,
  author       = {Hans{-}J{\"{u}}rgen Bandelt and
                  Victor Chepoi},
  title        = {Cellular Bipartite Graphs},
  journal      = {European. J. Combin.},
  volume       = {17},
  number       = {2-3},
  pages        = {121--134},
  year         = {1996},
  url          = {https://doi.org/10.1006/eujc.1996.0011},
  doi          = {10.1006/EUJC.1996.0011},
  bibsource    = {dblp computer science bibliography, https://dblp.org}
}

@article{BaCh_wma1,
  Author =	 {Bandelt, Hans-J\"urgen and Chepoi, Victor},
  Doi =		 {10.1016/j.ejc.2006.07.003},
  Fjournal =	 {European Journal of Combinatorics},
  Issn =	 {0195-6698},
  Journal =	 {European J. Combin.},
  Number =	 6,
  Pages =	 {1640--1661},
  Review =	 {\MR {2339492 (2008h:05038)}},
  Sauthor =	 {Bandelt, H.-J. and Chepoi, V.},
  Title =	 {The Algebra of Metric Betweenness {I}: Subdirect
                  Representation and Retraction},
  Volume =	 28,
  Year =	 2007
}

@article{BaMuWi,
  Author =	 {Bandelt, Hans-J{\"u}rgen and Mulder, Henry Martyn
                  and Wilkeit, Elke},
  Doi =		 {10.1002/jgt.3190180705},
  Fjournal =	 {Journal of Graph Theory},
  Issn =	 {0364-9024},
  Journal =	 {J. Graph Theory},
  Number =	 7,
  Pages =	 {681--703},
  Review =	 {\MR {1297190 (95h:05059)}},
  Sauthor =	 {Bandelt, H.-J. and Mulder, H. M. and Wilkeit, E.},
  Title =	 {Quasi-median graphs and algebras},
  Volume =	 18,
  Year =	 1994
}

@phdthesis{QM,
  Author =	 {Genevois, Anthony},
  School =	 {Aix-Marseille Universit\'e},
  Title =	 {Cubical-like geometry of quasi-median graphs and applications to geometric group theory},
  Year =	 2017
}

@article{DK25,
 author = {Delucchi, Emanuele and Knauer, Kolja},
 title = {Finitary affine oriented matroids},
 fjournal = {Discrete \& Computational Geometry},
 journal = {Discrete Comput. Geom.},
 issn = {0179-5376},
 volume = {73},
 number = {1},
 pages = {208--257},
 year = {2025},
 language = {English},
 doi = {10.1007/s00454-024-00651-z},
 zbMATH = {7966744}
}

@article{CKM20,
 author = {Chepoi, Victor and Knauer, Kolja and Marc, Tilen},
 title = {Hypercellular graphs: partial cubes without {{\(Q_3^-\)}} as partial cube minor},
 fjournal = {Discrete Mathematics},
 journal = {Discrete Math.},
 issn = {0012-365X},
 volume = {343},
 number = {4},
 pages = {28},
 note = {Id/No 111678},
 year = {2020},
 language = {English},
 doi = {10.1016/j.disc.2019.111678},
 zbMATH = {7170087},
 Zbl = {1433.05228}
}

@Book{BjLVStWhZi,
  Author =	 {Bj{\"o}rner, Anders and Las Vergnas, Michel and
                  Sturmfels, Bernd and White, Neil and Ziegler,
                  G{\"u}nter},
  Title =	 {Oriented {M}atroids.},
  Edition =	 {2nd ed.},
  FSeries =	 {Encyclopedia of Mathematics and Its Applications},
  Series =	 {Encycl. Math. Appl.},
  ISSN =	 {0953-4806},
  Volume =	 46,
  ISBN =	 {0-521-77750-X},
  Year =	 1999,
  Publisher =	 {Cambridge University Press},
  Language =	 {English},
  DOI =		 {10.1017/CBO9780511586507},
  zbMATH =	 1380576,
  Zbl =		 {0944.52006}
}

@Article{BlLV,
  FAuthor =	 {Bland, Robert G. and Las Vergnas, Michel},
  Author =	 {Bland, R. G. and Las Vergnas, M.},
  Title =	 {Orientability of matroids},
  FJournal =	 {Journal of Combinatorial Theory. Series B},
  Journal =	 {J. Comb. Theory, Ser. B},
  ISSN =	 {0095-8956},
  Volume =	 24,
  Number =	 1,
  Pages =	 {94--123},
  Year =	 1978,
  Language =	 {English},
  DOI =		 {10.1016/0095-8956(78)90080-1},
  zbMATH =	 3582165,
  Zbl =		 {0374.05016}
}

@article{Ch_CAT,
  Author =	 {Chepoi, Victor},
  Doi =		 {10.1006/aama.1999.0677},
  Fjournal =	 {Advances in Applied Mathematics},
  Issn =	 {0196-8858},
  Journal =	 {Adv. in Appl. Math.},
  Number =	 2,
  Pages =	 {125--179},
  Review =	 {\MR {1748966 (2001a:57004)}},
  Sauthor =	 {Chepoi, V.},
  Title =	 {Graphs of some {CAT}(0) Complexes},
  Volume =	 24,
  Year =	 2000
}

@book{Davis,
  Address =	 {Princeton, NJ},
  Author =	 {Davis, M. W.},
  Fpublisher =	 {Princeton University Press},
  Fseries =	 {London Mathematical Society Monographs Series},
  Isbn =	 {0-691-13138-4},
  Pages =	 {xvi+584},
  Place =	 {Princeton, NJ},
  Publisher =	 {Princeton Univ. Press},
  Review =	 {\MR {2360474 (2008k:20091)}},
  Sauthor =	 {Davis, Michael W.},
  Series =	 {London Math. Soc. Monogr. Ser.},
  Title =	 {The {G}eometry and {T}opology of {C}oxeter groups},
  Volume =	 32,
  Year =	 2008
}

@article{Dj,
  Author =	 {Djokovi\'{c}, Dragomir {\v{Z}}.},
  Doi =		 {10.1016/0095-8956(73)90010-5},
  Fjournal =	 {Journal of Combinatorial Theory, Series {B}},
  Journal =	 {J. Combin. Theory Ser. {B}},
  Number =	 3,
  Pages =	 {263--267},
  Sauthor =	 {Djokovi\'{c}, D. \v{Z}.},
  Title =	 {Distance-preserving Subgraphs of Hypercubes},
  Volume =	 14,
  Year =	 1973
}

@article{DrSch,
  Author =	 {Dress, Andreas W. M. and Scharlau, Rudolf},
  Doi =		 {10.1007/BF01840131},
  Fjournal =	 {Aequationes Mathematicae},
  Issn =	 {0001-9054},
  Journal =	 {Aequationes Math.},
  Number =	 1,
  Pages =	 {112--120},
  Review =	 {\MR {915878 (89c:54057)}},
  Sauthor =	 {Dress, A. W. M. and Scharlau, R.},
  Title =	 {Gated sets in metric spaces},
  Volume =	 34,
  Year =	 1987
}

@Article{FoLa,
  FAuthor =	 {Folkman, Jon and Lawrence, Jim},
  Author =	 {Folkman, J. and Lawrence, J.},
  Title =	 {Oriented matroids},
  FJournal =	 {Journal of Combinatorial Theory. Series B},
  Journal =	 {J. Comb. Theory, Ser. B},
  ISSN =	 {0095-8956},
  Volume =	 25,
  Number =	 2,
  Pages =	 {199--236},
  Year =	 1978,
  Language =	 {English},
  doi =		 {10.1016/0095-8956(78)90039-4},
  zbMATH =	 3508524,
  Zbl =		 {0325.05019}
}

@incollection{HaPau,
  Address =	 {Coventry},
  Author =	 {Haglund, Fr{\'e}d{\'e}ric and Paulin,
                  Fr{\'e}d{\'e}ric},
  Booktitle =	 {The Epstein Birthday Schrift},
  Doi =		 {10.2140/gtm.1998.1.181},
  Flanguage =	 {French, with English and French summaries},
  Fpublisher =	 {Mathematical Sciences Publishers},
  Fseries =	 {Geometry \& Topology Monographs},
  Pages =	 {181--248},
  Publisher =	 {Math. Sci. Publ.},
  Review =	 {\MR {1668359 (2000b:20034)}},
  Sauthor =	 {Haglund, F. and Paulin, F.},
  Series =	 {Geom. Topol. Monogr.},
  Title =	 {Simplicit\'e de Groupes d'Automorphismes d'Espaces
                  \`a Courbure N\'egative},
  Volume =	 1,
  Year =	 1998
}

@book{Hat,
  Address =	 {Cambridge},
  Author =	 {Hatcher, Allen},
  Fpublisher =	 {Cambridge University Press},
  Isbn =	 {0-521-79540-0},
  Pages =	 {xii+544},
  Place =	 {Cambridge},
  Publisher =	 {Cambridge Univ. Press},
  Review =	 {\MR {1867354 (2002k:55001)}},
  Sauthor =	 {Hatcher, A.},
  Title =	 {Algebraic Topology},
  Year =	 2002
}

@Misc{Knauer_HDR,
  Author =	 {Knauer, Kolja},
  Sauthor =	 {Knauer, Kolja},
  Title =	 {Oriented matroids and beyond: complexes, partial cubes, and corners},
  HowPublished={Habilitation Thesis, Aix-Marseille Université},
  Pages={215 pages},
  Year ={2021},
}

@Article{KnMa,
  FAuthor =	 {Knauer, Kolja and Marc, Tilen},
  Author =	 {Knauer, K. and Marc, T.},
  Title =	 {On tope graphs of complexes of oriented matroids},
  FJournal =	 {Discrete \& Computational Geometry},
  Journal =	 {Discrete Comput. Geom.},
  ISSN =	 {0179-5376},
  Volume =	 63,
  Number =	 2,
  Pages =	 {377--417},
  Year =	 2020,
  Language =	 {English},
  DOI =		 {10.1007/s00454-019-00111-z},
  zbMATH =	 7160869,
  Zbl =		 {1431.05034}
}

@book{Mu,
 author = {Mulder, H. M.},
 title = {The interval function of a graph},
 fseries = {Mathematical Centre Tracts},
 series = {Math. Cent. Tracts},
 volume = {132},
 year = {1980},
 publisher = {Centrum voor Wiskunde en Informatica (CWI), Amsterdam},
 language = {English},
 zbMATH = {3697163},
 Zbl = {0446.05039}
}

@article {MR1663779,
    AUTHOR = {Gerasimov, V.},
     TITLE = {Fixed-point-free actions on cubings [translation of {\it
              {A}lgebra, geometry, analysis and mathematical physics
              ({R}ussian) ({N}ovosibirsk, 1996)}, 91--109, 190, {I}zdat.
              {R}oss. {A}kad. {N}auk {S}ibirsk. {O}tdel. {I}nst. {M}at.,
              {N}ovosibirsk, 1997; {MR}1624115 (99c:20049)]},
   JOURNAL = {Siberian Adv. Math.},
  FJOURNAL = {Siberian Advances in Mathematics},
    VOLUME = {8},
      YEAR = {1998},
    NUMBER = {3},
     PAGES = {36--58},
      ISSN = {1055-1344,1934-8126},
   MRCLASS = {20F32 (57M07)},
  MRNUMBER = {1663779},
}

@techreport{Ro,
  Author =	 {Roller, Martin},
  Institution =	 {Univ. of Southampton},
  Sauthor =	 {Roller, M.},
  Title =	 {Poc sets, Median Algebras and Group Actions},
  Year =	 1998
}

@Book{Gru03,
 Author = {Gr{\"u}nbaum, B.},
 Title = {Convex {P}olytopes. {Prepared} by {Volker} {Kaibel}, {Victor} {Klee}, and {G{\"u}nter} {M}. {Ziegler}},
 Edition = {2nd ed.},
 FSeries = {Graduate Texts in Mathematics},
 Series = {Grad. Texts Math.},
 ISSN = {0072-5285},
 Volume = {221},
 ISBN = {0-387-40409-0},
 Year = {2003},
 Publisher = {New York, NY: Springer},
 Language = {English},
 zbMATH = {2008526},
 Zbl = {1033.52001}
}

@inproceedings{St,
  Author =	 {Stark, Eugene W.},
  Booktitle =	 {MFPS 1989},
  Fbooktitle =	 {Mathematical Foundations of Programming Semantics
                  1989},
  Fseries =	 {Lecture Notes in Computer Science},
  Pages =	 {53--79},
  Publisher =	 {Springer},
  Sauthor =	 {Stark, E. W.},
  Series =	 {Lecture Notes in Comput. Sci.},
  Title =	 {Connections between a Concrete and an Abstract Model
                  of Concurrent Systems},
  Volume =	 442,
  Year =	 1989
}

@book{Ti,
 author = {Tits, Jacques},
 title = {Buildings of spherical type and finite {BN}-pairs},
 fseries = {Lecture Notes in Mathematics},
 series = {Lect. Notes Math.},
 issn = {0075-8434},
 volume = {386},
 year = {1974},
 publisher = {Springer, Cham},
 language = {English},
 zbMATH = {3462178},
 Zbl = {0295.20047}
}

@book{Zi,
  Address =	 {New York},
  Author =	 {Ziegler, G{\"u}nter M.},
  Doi =		 {10.1007/978-1-4613-8431-1},
  Fseries =	 {Graduate Texts in Mathematics},
  Isbn =	 {0-387-94365-X},
  Pages =	 {x+370},
  Publisher =	 {Springer--Verlag},
  Review =	 {\MR {1311028 (96a:52011)}},
  Sauthor =	 {Ziegler, G. M.},
  Series =	 {Grad. Texts in Math.},
  Title =	 {Lectures on Polytopes},
  Volume =	 152,
  Year =	 1995
}

@PhdThesis{ed-ma-82,
  author =	 {Edmonds, J. and Mandel, A.},
  title =	 {Topology of Oriented Matroids},
  school =	 {University of Waterloo},
  year =	 1982,
  note =	 {PhD thesis of A. Mandel, 333 pages},
}

@article{Ch_Hamming,
author={Chepoi, Victor},
title={Isometric subgraphs of {H}amming graphs and $d$-convexity},
journal={Cybernetics (Kiev)},
volume={1},
year={1988},
pages={6--10},
}

@article{Wi,
author={Wilkeit, Elke}, 
title={Isometric embedding in {H}amming graphs},
journal={J. Combin. Th. Ser. B}, 
volume={50},
year={1990},
pages={179–197},
}
\end{document}